\documentclass[a4paper, 12pt]{amsart}
\usepackage{amsmath,amsfonts,amssymb,mathrsfs}
\usepackage{latexsym, amscd, epsfig}
\usepackage[utf8]{inputenc}
\usepackage{color}
\usepackage{amsmath}
\usepackage{amssymb}
\usepackage{amsthm}
\usepackage{graphicx}
\usepackage{esint}
\usepackage[colorlinks=true,linkcolor=blue]{hyperref}
\usepackage{verbatim}
\usepackage{accents}
\newcommand{\ubar}[1]{\underaccent{\bar}{#1}}

\theoremstyle{plain}
\newtheorem{theorem}{Theorem}
\newtheorem{proposition}{Proposition}[section]
\newtheorem{lemma}[proposition]{Lemma}
\newtheorem{corollary}[proposition]{Corollary}

\newtheorem{remark}[proposition]{Remark}

\theoremstyle{problem}
\newtheorem {problem}{Problem}

\numberwithin{equation}{section}
\def\Xint#1{\mathchoice
	{\XXint\displaystyle\textstyle{#1}}
	{\XXint\textstyle\scriptstyle{#1}}
	{\XXint\scriptstyle\scriptscriptstyle{#1}}
	{\XXint\scriptscriptstyle\scriptscriptstyle{#1}}
	\!\int}
\def\XXint#1#2#3{{\setbox0=\hbox{$#1{#2#3}{\int}$ }
		\vcenter{\hbox{$#2#3$ }}\kern-.6\wd0}}

\def\dashint{\Xint-}

\def\a{\alpha}
\def\b{\beta}
\def\c{\cdot}
\def\d{\delta}

\def\g{\gamma}
\def\G{\Gamma}

\def\ld{\lambda}
\def\Ld{\Lambda}
\def\n{\nabla}

\def\O{\Omega}

\def\bp{\mathbf{p}}
\def\pt{\partial}
\def\q{\quad}
\def\r{\rho}
\def\th{\theta}
\def\ul{\ubar}
\def\ur{\Upsilon}
\def\v{\epsilon}
\def\vp{\varphi}
\def\B{\mathcal{B}}
\def\D{\mathcal{D}}

\def\MG{\mathcal{G}}
\def\h{\mathfrak{h}}
\def\J{\mathcal{J}}
\def\K{\mathcal{K}}
\def\N{\mathcal{N}}
\def\MO{\mathcal{O}}
\def\bp{\mathbf{p}}

\def\R{\mathbb{R}}

\def\t{\mathfrak{t}}
\def\mn{\mathbf{n}}
\def\u{\mathbf{u}}

\usepackage{xspace}

\begin{document}

\title[three-dimensional axisymmetric jet flows with large vorticity]{\bf The jet problem for three-dimensional axisymmetric compressible subsonic flows with large vorticity}

\author{Yan Li}
\address{School of Mathematical Sciences, Shanghai Jiao Tong University, 800 Dongchuan Road, Shanghai, 200240, China}
\email{liyanly@sjtu.edu.cn}

\begin{abstract}
	In this paper, we establish the existence of three-dimensional axisymmetric compressible jet flows for steady Euler system with large vorticity by using the variational method. More precisely, for given axial velocity of the flow at the upstream, if the mass flux is sufficiently large, we can find a unique outer pressure such that a smooth subsonic three-dimensional axisymmetric jet flow with large vorticity exists and has certain far fields behavior.
\end{abstract}

\keywords{axisymmetric flows, Euler system, free boundary, subsonic jet, vorticity.}
\subjclass[2010]{%\AMSMOS
	35Q31, 35R35, 35J20, 35J70, 35M32, 76N10}
\maketitle

%%%%%%%%%%%%%%%%%%%%%%%%%%%%%%%%%%%%%%%%%%%%%%%%%%%%%%%%%%%%%%%
\section{Introduction and main results}
In the 1980s, Alt, Caffarelli and Friedman developed a systematic variational method to study the Bernoulli type free boundary problems such as jets and cavities (\cite{AC81,ACF84,ACF85,F82_book_variational}).  
Based on this variational method, various physical models have been studied (cf. \cite{CDW2017,CDX2019,CDX2020_Rethy,CDZ2021,DW2021_two_axi,LSTX2023,L2023,L2023_axi_small} and references therein). 
In particular, the well-posedness of two-dimensional subsonic jet flows for steady Euler system with large vorticity was investigated in \cite{LSTX2023,L2023}. The existence of three-dimensional axisymmetric subsonic jet flows for steady full Euler system with nonzero vorticity was proved in \cite{L2023_axi_small}. In this paper, we study three-dimensional axisymmetric subsonic jet flows for steady Euler system with large vorticity.  

Three-dimensional steady isentropic ideal flows are governed by the following Euler system
\begin{equation}\label{Euler origional}
\begin{cases}\begin{split}
& \n\cdot(\r \mathbf{u})=0,\\
& \n \cdot(\r \mathbf{u}\otimes\mathbf{u})+\n p=0,
\end{split}\end{cases}
\end{equation}
where $\mathbf u=(u_1,u_2,u_3)$ denotes the flow velocity, $\r$ is the density, and $p$ is the pressure of the flow. For the polytropic gas, after nondimensionalization, the equation of state can be written as  $ p(\rho)=\rho^{\gamma}/\gamma$, where the constant $\gamma>1$ is called the adiabatic exponent. The local sound speed and the Mach number of the flow are defined as
\begin{equation}\label{eq:sound}
	c(\rho)=\sqrt{p'(\rho)}=\rho^{\frac{\gamma-1}{2}}\quad \text{and}\quad M=\frac{|\textbf{u}|}{c(\rho)},
\end{equation}
respectively. 
The flow is called subsonic if $M<1$, sonic if $M=1$, and supersonic if $M>1$, respectively.

We specialize to axisymmetric compressible flows and take the symmetry axis to be the $x_1$-axis. Denote $x=x_1$ and $y=\sqrt{x_2^2+x_3^2}$. Let $u$ and $v$ be the axial velocity and the radial velocity of the flow,  respectively. Assume that the swirl velocity of the flow is zero, then
\begin{equation*}
	\begin{cases}\begin{split}
			&u_1(x_1,x_2,x_3)=u(x,y),\\
			&u_2(x_1,x_2,x_3)=v(x,y)\frac{x_2}{y},\\
			&u_3(x_1,x_2,x_3)=v(x,y)\frac{x_3}{y}.
	\end{split}\end{cases}
\end{equation*}
%$$u_1(x,x_2,x_3)=u(x,y),\q u_2(x,x_2,x_3)=v(x,y)\frac{x_2}{y} \q {\rm and}\q u_3(x,x_2,x_3)=v(x,y)\frac{x_3}{y}.$$
Hence in cylindrical coordinates, the Euler system \eqref{Euler origional} can be rewritten as
\begin{equation}\label{Euler system}
	\begin{cases}\begin{split}
			&\n\c(y\r \u)=0,\\
			&\n\c(y\r \u\otimes \u)+y\n p=0,
	\end{split}\end{cases}
\end{equation}
where $\u=(u,v)$ and $\n=(\pt_x,\pt_y)$.
%If the flow is away from vacuum, it follows from \eqref{Euler system} that
%\begin{equation}\label{B equations}
%	\u\c\n B=0,
%\end{equation}
%where $B$ is the Bernoulli function defined by
%\begin{equation}\label{B def}
%	B=\frac{|\u|^2}{2}+\frac{\r^{\gamma-1}}{\g-1}.
%\end{equation}

Now given an axisymmetric nozzle in $\R^3$ with the symmetry axis and the upper solid boundary of the nozzle expressed as
\begin{equation}\label{nozzle}
	\N_0:=\{(x, 0): x\in \mathbb{R}\}
	\quad \text{and}\quad	
	\N:=\{(x,y): x=N(y),\, y\in[1,\bar H)\},
\end{equation}
respectively, where $\bar{H}>1$, $N\in C^{1,\bar\a}([1,\bar{H}])$ for some  $\bar\a\in(0,1)$ and it satisfies
\begin{equation}\label{nozzle condition}
	N(1)=0 \q {\rm and} \q \lim_{y\to\bar{H}-}N(y)=-\infty.
\end{equation}
%\textcolor{red}{In addition, we assume there exists $h_*\in[1,\bar H)$ such that the nozzle is monotone decreasing for $y\in(h_*,\bar H)$, i.e.,
%\begin{equation}\label{nozzle_mono}
%N'(y)\leq 0,\quad\text{for } y\in(h_*,\bar H).	
%\end{equation}}

The main goal of this paper is to study the following jet problem.

\begin{problem}\label{probelm 1}
	Given a mass flux $Q>0$ and a positive axial velocity $\bar{u}(y)$ of the flow at upstream $x\to-\infty$, find $(\r,\u)$, the free boundary $\G$, and the outer pressure $\ul{p}$ which is assumed to be a constant, such that the following statements hold. 
	\begin{enumerate}
		\item[\rm (1)] The free boundary $\G$ joins the nozzle boundary $\N$ as a continuous surface and tends asymptotically horizontal at downstream as $x\to\infty$.
		
		\item[\rm (2)] The solution $(\r,\u)$ solves the Euler system \eqref{Euler system} in the flow region $\MO$ bounded by $\N_0$, $\N$, and $\Gamma$. It takes the incoming data at the upstream, i.e.,
		\begin{equation}\label{upstream condition}
			u(x,y)\to\bar u(y)  \q \text{as } x\to-\infty,
		\end{equation}
		and
		\begin{equation}\label{Q}
			\int_0^1 y(\r u)(0,y)dy=Q.
		\end{equation}
		Moreover, it satisfies the boundary conditions 
		\begin{equation*}
			p(\rho)=\ul{p} \text{ on } \G, \q\text{and}\q \u\c\mn=0  \text{ on } \N\cup \G,
		\end{equation*}
		where $\mn$ is the unit normal along $\N\cup\G$.
	\end{enumerate}
\end{problem}

The main results in this paper can be stated as follows.

\begin{theorem}\label{result}
	Let the nozzle boundary $\N$ defined in \eqref{nozzle} satisfy \eqref{nozzle condition} % and \eqref{nozzle_mono}, 
	and $Q>0$ be a constant. Assume that $\bar{u}\in C^{1,1}([0,\bar{H}])$ satisfies 
	\begin{equation}\label{ubar_condition}\begin{split}
			&\inf_{y\in[0,\bar H]}\bar u(y)>0,\q \lim_{y\to0}\frac{\bar u'(y)}{y}=0,\q
			\text{and }  \q
			0\leq \frac1y\left(\frac{\bar u'(y)}y\right)'<\infty  
			\text{ for } y\in(0,\bar H].
	\end{split}\end{equation}
	%for some constant $C>0$. 
	Then there exists a constant $Q^*=Q^*(\bar u,\g,\N)>0$ such that for any $Q>Q^*$, 
	there are functions $\rho,\mathbf{u}\in C^{1,\alpha}(\mathcal{O})\cap C^0(\overline{\mathcal{O}})$ (for any $\alpha\in(0,1)$) where $\mathcal O$ is the flow region, the free boundary $\G$, and the outer pressure $\ubar p$ such that $(\rho,\mathbf u,\Gamma,\ul p)$ solves Problem \ref{probelm 1}.  Furthermore, the following properties hold.
	\begin{itemize}
	\item[(i)] (Smooth fit) The free boundary $\G$ joins the nozzle boundary $\N$ as a $C^1$ surface.
	
	\item[(ii)] The free boundary $\Gamma$ is given by a graph $x=\Upsilon(y)$, $y\in (\ubar H, 1]$, where $\Upsilon$ is a $C^{2,\alpha}$ function, $\ubar H\in(0,1)$, and $\lim_{y\rightarrow \ubar H+} \Upsilon(y)=\infty$. For $x$ sufficiently large, the free boundary  can also be written as $y=f(x)$ for some $C^{2,\alpha}$ function $f$ which satisfies
		$$\lim_{x\rightarrow \infty}f(x)=\ubar H
		\quad\text{and}\quad  \lim_{x\rightarrow \infty}f'(x)=0.$$
	
	\item[(iii)] The flow is globally uniformly subsonic and has negative vertical velocity in the flow region $\MO$, i.e.,
	\begin{equation*}
		\sup_{\overline{\MO}}\frac{|\u|^2}{c^2}<1 \q\text{and}\q v<0 \ \text{in }\MO.
	\end{equation*}
	
		\item[(iv)] (Upstream and downstream asymptotics)
		%The free boundary $\Gamma$ is asymptotically horizontal at downstream $x\rightarrow \infty$. Let $\ul{H}$ be the asymptotic height of the free boundary at the downstream.
		There exist positive constants $\bar\rho$ and $\ubar{\rho}$, which are the upstream and downstream density respectively, and a positive function $\ubar {u}\in C^{1,\alpha}((0,\ubar H])$, which is the downstream axial velocity, such that 
		\begin{equation}\label{asymptotic upstream}
			\|(\r,\u)(x,\c)-(\bar{\r},\bar{u}(\c),0)\|_{C^{1,\alpha}((0,\bar{H}))}\to0, \q\text{as } x\to-\infty
		\end{equation}
		and
		\begin{equation}\label{asymptotic downstream}
			\|(\r,\u)(x,\c)-(\ul{\r},\ul{u}(\c),0)\|_{C^{1,\alpha}((0,\ul{H}))}\to0, \q\text{as } x\to\infty.
		\end{equation}
		Moreover, 
		\begin{align}\label{def_pbar_rhobar}
			\bar\r=\frac{Q}{\int_{0}^{\bar H}y\bar u(y)dy} 	\q\text{and}\q
			\ul\r=(\g \ul p)^{\frac1\g};
		\end{align}
		the downstream axial velocity $\ul{u}$ and height $\ul{H}>0$ are also  uniquely determined by $Q$, $\bar{u},\,\g,\,\bar H$, and $\ul p$.
	
		\item[(v)] (Uniqueness of the outer pressure) The outer pressure $\ubar p$ such that the solution $(\rho,\mathbf u,\Gamma)$ satisfies the properties (i)-(iv) is uniquely determined by $Q$, $\bar{u},\,\g$, and $\bar H$.
		%\item[\rm (ii)] \textcolor{red}{(Uniqueness)} The Euler flow which satisfies all properties in part (i) is unique.
		%\item[\rm (iii)] \textcolor{red}{(Critical mass flux)} $Q_c$ is the upper critical mass flux for the existence of subsonic jet flow in the following sense: either
		%$$\sup_{\overline{\MO}}(|\u|^2-c^2)\to0 \q\text{as } Q\to Q_c,$$ or there is no $\sigma>0$ such that for all $Q\in(Q_c, Q_c+\sigma)$, there are Euler flows satisfying all properties in part (i) and
		%$$\sup_{Q\in(Q_c,Q_c+\sigma)}\sup_{\overline{\MO}}(|\u|^2-c^2)<0.$$
	\end{itemize}
\end{theorem}

\begin{remark}
The vorticity $\omega$ of the flow at the upstream is actually $-\bar u'$,  and the quantity $\omega/(y\rho)$ is a constant along each streamline (cf. the third equation in \eqref{eq:euler}). Therefore,  the vorticity could be large under the conditions of $\bar u$ in \eqref{ubar_condition}. 
%In addition, the conditions in \eqref{ubar_condition} actually indicates that the quantity  $\omega/(y\rho)$ and its derivative have negative sign at the upstream.
This is the main difference between Theorem \ref{result} and \cite[Theorem 1]{L2023_axi_small}.
\end{remark}

\begin{remark}
For simplicity we consider isentropic jet flows in Theorem \ref{result}.  
The arguments in this paper also hold for three-dimensional axisymmetric non-isentropic jet flows with large vorticity. 
\end{remark}

The main idea of the proof for Theorem \ref{result} is essentially the same as that in \cite{L2023_axi_small}. From \cite{L2023_axi_small} we know  that, in the subsonic state, the Euler system can be reduced to an elliptic equation of the stream function, which is an Euler-Lagrange equation of an energy functional. Then we can describe the jet problem in terms of the stream function. In order to solve the jet problem, we introduce the subsonic and domain truncations, and then give appropriate boundary conditions for the truncated problems. The truncated problems have a  variational structure (cf. Section \ref{sec variational formulation}). With the help of the framework in \cite{ACF84,ACF85} and techniques in  \cite{LSTX2023,L2023_axi_small}, we obtain the existence of minimizers for the corresponding variational problems, the regularity of the free boundaries, and other good properties of the solutions. After removing the domain and subsonic truncations, we prove the existence of subsonic solutions to the original jet problem. Finally, using a shifting technique, we show that the outer pressure that makes the subsonic solutions satisfy all properties in Theorem \ref{result}, especially the far field behavior \eqref{asymptotic downstream}, is unique.

Compared with \cite{L2023_axi_small}, the difficulties in proving Theorem \ref{result} arise from the inhomogeneous term of the elliptic equation satisfied by the stream function (cf. Lemma \ref{EL}), which is not small as in \cite{L2023_axi_small}. Fortunately, under the conditions of $\bar u$ in \eqref{ubar_condition}, we can get the non-negativity of the inhomogeneous term and certain structural conditions of the elliptic equation (Proposition \ref{Gproperties pro}). 
The non-negativity of the inhomogeneous term (cf. \eqref{supportG}) enables us to derive a uniform estimate of solutions to the truncated problems (cf. Lemma \ref{psi0 sup-subsol} and Proposition \ref{psi monotonic}), which is crucial for avoiding the singularity of solutions near the symmetry axis. The structural conditions \eqref{det_0} allow us to prove the asymptotic behavior of solutions by energy estimates (cf. Proposition \ref{pro_asymptotic_upstream}) and prove the uniqueness of the outer pressure by the comparison principle (cf. Proposition \ref{uniqueness pro}).

The rest of the paper is organized as follows. In Section \ref{sec stream formulation}, we reformulate the Euler system and the jet problem in terms of the stream function. In Section \ref{sec variational formulation}, we give the variational formulation for the jet problem after the domain and subsonic truncations. In Section \ref{sec existence and regularity}, the existence and regularity of solutions to the truncated free boundary problems are obtained. Fine properties of the solutions such as the monotonicity property, the continuous fit and smooth fit of the free boundary, and some uniform estimates are established in Section \ref{sec fine properties}. In Section \ref{sec remove truncation}, we remove the   the domain and subsonic truncations, and then prove the far fields behavior of the solutions. Section \ref{sec uniqueness} is devoted to the uniqueness of the outer pressure.

In the rest of the paper, $(\pt_1,\pt_2)$ represents $(\pt_x,\pt_y)$, and for convention the repeated indices mean the summation.

\section{Stream function formulation and subsonic truncation}\label{sec stream formulation}

In this section, we use the stream function to reduce the Euler system \eqref{Euler system} to a {second order} quasilinear equation, which is elliptic in the subsonic region and becomes singular at the sonic state. Then we describe the jet problem in terms of the stream function. One of the main difficulties to solve the jet problem comes from the possible degeneracy of the quasilinear equation near the sonic state. Thus we introduce a subsonic truncation such that the equation is always uniformly elliptic after the truncation. 

\subsection{The equation for the stream function} 
%Inspired by \cite{XX2010}, w
We first give an equivalent formulation for the compressible Euler system. 
\begin{proposition}\label{prop:Euler_equivalent}
	Let ${\tilde\MO} \subset \mathbb{R}^2$ be the domain bounded by two streamlines $\N_0$ (the $x$-axis) and
	\[
	{\tilde \N} :=\{(x, y): x={\tilde N(y)}, \ubar H< y< \bar H\},\]
	where  $0<\ubar H< \bar H<\infty$ and $\tilde N:(\ubar H, \bar H)\rightarrow \R$  is a $C^1$ function with
	\[
	\lim_{y\to \bar H-}\tilde N(y) =-\infty\quad\text{and}\quad  \lim_{y\to \ubar H+} \tilde N(y) =\infty.
	\]
	Let $\rho:\overline{\tilde \MO}\rightarrow (0,\infty)$ and $\u=(u,v):\overline{\tilde\MO}\rightarrow \R^2$ be $C^{1,1}$ in $\tilde\MO$ and continuous up to $\pt\tilde\MO$ except finitely many points. Suppose that  $\u$ satisfies the slip boundary condition $\u\cdot \mn=0$ on $\partial\tilde\MO$, $(\rho, \u)$ satisfies the upstream asymptotics \eqref{asymptotic upstream} with a positive constant $\bar\rho$ and a positive function $\bar u\in C^{1,1}([0,\bar H])$. Moreover, suppose that
	\begin{equation}\label{v_negative}
		v< 0\quad \text{ in } \tilde\MO.
	\end{equation}
	Then $(\rho, \u)$ solves the Euler system \eqref{Euler system} in $\tilde\MO$ if and only if $(\rho, \u)$ satisfies
	\begin{equation}\label{eq:euler}
		\left\{
		\begin{aligned}
			&\nabla\cdot(y\rho \u)=0, \\
			&\u \cdot \nabla \mathscr{B}(\rho, \u) =0, \\
			&\u\cdot \nabla \left(\frac{\omega}{y\rho}\right) =0,
		\end{aligned}
		\right.
	\end{equation}
	where
	\begin{align}\label{def:euler}
		\mathscr{B}(\rho, \u):= \frac{|\u|^2}{2}+h(\rho),\quad \omega:=\pt_{x}v-\pt_{y}u,\quad\text{and}\quad h(\rho):=\frac{\rho^{\gamma-1}}{\gamma-1}
	\end{align}
	are the Bernoulli function, the vorticity, and the enthalpy of the flow, respectively.
\end{proposition}
The proof of Proposition \ref{prop:Euler_equivalent} is essentially the same as that in \cite[Proposition 2.1]{LSTX2023} and \cite[Proposition 2.1]{L2023_axi_small}, so we omit it here. Note that the condition \eqref{v_negative} assures that through each point in the domain $\tilde \MO$ there is one and only one streamline, which is globally defined in $\tilde\MO$.  

It follows from the continuity equation  (the first  equation in \eqref{eq:euler}) that there is a stream function $\psi$ satisfying
\begin{equation}\label{psi gradient}
	\n\psi=(-y\r v, y\r u).
\end{equation}
Besides, the Bernoulli law (the second equation in \eqref{eq:euler}) implies $\mathscr{B}(\rho,\mathbf{u})$ is conserved along each streamline. Now we show that $\mathscr{B}(\rho,\mathbf{u})$ can be expressed as a function of the stream function $\psi$ under the assumptions of Proposition \ref{prop:Euler_equivalent}.
\begin{proposition}\label{BS pro}
	Let an axial velocity $\bar u\in C^{1,1}([0,\bar H])$ satisfy \eqref{ubar_condition} and a mass flux $Q>0$. Suppose that $(\rho,
	\u)$ is a solution to the Euler system \eqref{Euler system} and satisfies the assumptions of Proposition \ref{prop:Euler_equivalent}. Then there is a function $\B:[0,Q]\to\R$, $\B\in C^{1,1}([0,Q])$ such that 
	\begin{equation}\label{Bpsi relation}
		\mathscr{B}(\r,\u)=\frac{|\n\psi|^2}{2y^2\r^2}+h(\r)=\B(\psi) \q \text{in }\tilde\MO.
	\end{equation}
	Denote
	\begin{equation}\label{k0}
		\kappa_0:=\|\bar u''\|_{L^{\infty}((0, \bar H])}+
		\left\|\frac1y\left(\frac{\bar u'}y\right)'\right\|_{L^{\infty}((0, \bar H])}.
	\end{equation}
	Then
	%\begin{equation}\label{eq:u0_eps0_B}
	%	\|\mathcal{B}'\|_{L^\infty([0, Q])}\leq \frac{\kappa_0}{\bar\rho} \quad{and}\quad
	%	\|\mathcal{B}''\|_{L^\infty([0,Q])}\leq \frac{\kappa_0}{\bar\rho^2\min_{y\in[0,\bar H]}\bar u(y)},
	%\end{equation}
\begin{equation}\label{eq:u0_eps0_B}
	0\leq \mathcal{B}'(z)\leq  \frac{\kappa_0}{\bar\rho} \quad{and}\quad
	0\leq \mathcal{B}''(z)\leq \frac{\kappa_0}{\bar\rho^2\inf_{y\in[0,\bar H]}\bar u(y)},\q\text{for } z\in (0,Q],
\end{equation}
	where
	\begin{equation}\label{eq:rhobar}
		\bar \rho =\frac{Q}{\int_0^{\bar H} y\bar u(y) dy}.
	\end{equation}
\end{proposition}
\begin{proof}
	In view of Proposition \ref{prop:Euler_equivalent}, the Bernoulli function $\mathscr{B}(\rho,\mathbf{u})$ is conserved along each streamline, which is globally well-defined in the flow region $\tilde\MO$. In particular, $\mathscr{B}(\rho, \mathbf{u})$ is uniquely determined by its value at the upstream. Let $\h(\psi;\bar\rho):[0,Q]\rightarrow [0,\bar H]$ be the position of the streamline at upstream where the stream function has the value $\psi$, i.e.,
	\begin{align}\label{h def}
		\psi=\bar\rho\int_0^{\h(\psi;\bar \rho)}y\bar u(y)\ dy,
	\end{align}
	where the upstream density $\bar\rho$ is uniquely determined by $Q$ and $\bar u$ from the upstream asymptotics \eqref{asymptotic upstream}, i.e. $\bar\rho$ satisfies \eqref{eq:rhobar}. The function $\h$ is well-defined as $\bar u>0$, moreover, differentiating \eqref{h def} with respect to $\psi$ one has 
	$$\h'(\psi;\bar\r)=\frac1{\bar\r\h(\psi;\bar\r)\bar u(\h(\psi;\bar\r))}.$$
	Since the Bernoulli function at the upstream is
	\begin{equation}\label{def:B_upstream}
		B(y):=\lim_{x\rightarrow -\infty} \mathscr{B}(\rho,\mathbf{u})(x,y)=\frac{\bar u^2(y)}{2}+h(\bar\rho) \quad\text{for }y\in[0,\bar H],
	\end{equation}
    using \eqref{psi gradient} one has 
	\begin{align}\label{eq:psi}
		\frac{|\nabla\psi|^2}{2y^2\rho^2}+h(\rho)=\mathcal{B}(\psi) \ \text{ in } \tilde\MO,
	\end{align}
	where the function $\mathcal B$ is defined as
	\begin{equation}\label{defB}
		\mathcal{B}(z):=B(\h(z;\bar\rho))=\frac{\bar u^2(\h(z;\bar\rho))}{2}+h(\bar\rho) \quad\text{for } z\in [0,Q].
	\end{equation}
	The straightforward computations give
	\begin{align}\label{eq:dB}
		\mathcal{B}'(z)=\frac{\bar u'(\h(z;\bar\rho))}{\bar\rho\h(z;\bar\r)} \quad\text{and}\quad  \mathcal{B}''(z)=\frac1
		%{\bar u''(\h(z;\bar\rho))\h(z;\bar\rho)-\bar u'(\h(z;\bar\rho))}
		{\bar\r^2\h^2(z;\bar\r)\bar u(\h(z;\bar\rho))}
		\left(\bar u''(\h(z;\bar\rho))-\frac{\bar u'(\h(z;\bar\rho))}{\h(z;\bar\rho)}\right).
	\end{align}
	Note that the conditions of $\bar u$ in \eqref{ubar_condition} imply $y(\bar u'(y)/y)'\geq0$. Hence $0\leq\bar u'(y)/y\leq \bar u''(y)$. Then  \eqref{eq:u0_eps0_B} follows directly from \eqref{eq:dB}, \eqref{ubar_condition}, and the definition of $\kappa_0$ in \eqref{k0}.
\end{proof}

For later purpose we extend the Bernoulli function $\mathcal{B}$ from $[0,Q]$ to $\R$ as follows:
firstly in view of \eqref{ubar_condition}, $\bar u$ can be extended to a $C^{1,1}$ function defined on  $\R$ (still denoted by $\bar u$) such that
\[
\bar u>0 \text{ on } \R, \quad {\bar u'= 0 \text{ on } (-\infty,0],}
\quad \bar u'\geq 0  \text{ on } [\bar H,\infty).
\]
Furthermore, the extension can be made such that
\begin{align}\label{ubar_extension}
	0<\bar u_\ast=\inf_{\R}\bar u\leq \bar u^*\leq \sup_{\R} \bar u =:\tilde u^\ast <\infty \quad\text{with} \quad \bar u_*:=\inf_{[0,\bar H]}\bar u, \quad \bar u^*:=\sup_{[0,\bar H]}\bar u,
\end{align}
where $\tilde u^\ast$ depends on $\bar u^*$ and $\|\bar u\|_{C^{1,1}([0,\bar H])}$, and
$$\|\bar u\|_{C^{1,1}(\R)}\leq C\|\bar u\|_{C^{1,1}([0,\bar H])}.$$
%$\xi(x_2)$ for $x_2>\bar H$, where $\xi(x_2)$ is a smooth decreasing function in $[\bar H,\infty)$ such that
%$$|\xi|\leq 1 \text{ in }[\bar H,\infty),\quad \xi=0 \text{ for } x_2\in [\bar H+1,\infty) \quad\text{and}\quad \xi'(\bar H)=\frac{\bar u'(\bar H)}{\bar u(\bar H)}.?$$
Consequently, using \eqref{defB} one naturally gets an extension of $\mathcal B$ to a $C^{1,1}$ function in $\R$ (still denoted by $\mathcal{B}$), which satisfies
\begin{equation}\label{eq:sign_B}
	{\mathcal{B}'(z)= 0 \text{ on } (-\infty, 0]}
	\quad \text{ and } \quad \mathcal{B}'(z)\geq 0 \text{ on } [Q,\infty).
\end{equation}
The function $\mathcal{B}$ is bounded from above and below, i.e.,
\begin{align}\label{eq:B}
	0<{B}_*\leq \mathcal{B}(z)\leq {B}^*<\infty,\quad z\in \R,
\end{align}
where $B_*$ and $B^*$ are defined as
\begin{equation}\label{defB*}
	B_*:=\frac12(\bar u_*)^2+h(\bar\rho),\quad  B^*:=\frac12(\tilde u^*)^2+h(\bar\rho).
\end{equation}
Moreover, 
\begin{equation*}\label{eq:B_derivative}
	%\bar \rho =\frac{Q}{\int_0^{\bar H} \bar u(s) ds}, \quad\text{and}\quad
	\|\mathcal{B}'\|_{L^\infty(\R)}+\|\mathcal{B}''\|_{L^\infty(\R)}\leq C (\|\mathcal{B}'\|_{L^\infty((0, Q])}+\|\mathcal{B}''\|_{L^\infty((0, Q])})
\end{equation*}
where $C$ is independent of $\bar\rho$.
 
Let us digress for the study on the flow state with the given Bernoulli constant.
For the flow state with given Bernoulli constant $s$, the flow
density $\rho$ and flow speed $q$ satisfy
\begin{equation*}
	\frac{q^2}{2}+h(\rho)=s.
\end{equation*}
Therefore, the speed $q$ satisfies
\begin{equation*}
	q=\mathfrak q(\rho,s)= \sqrt{2(s-h(\rho))}.
\end{equation*}
For each {fixed $s$}, we denote the critical density and the maximum density by
\begin{equation}\label{defrhoc}
	\varrho_c(s):=\left\{\frac{2(\gamma-1)}{\gamma+1} s\right\}^{\frac{1}{\gamma-1}}\quad\text{and}\quad   \varrho_m(s) :=\left\{(\gamma-1)s\right\}^{\frac{1}{\gamma-1}},
\end{equation}
respectively. 
For states with given Bernoulli constant $s$, note that {$s-h(\rho)\geq 0$ for $\rho\leq \varrho_m(s)$. Thus the flow state $q=\mathfrak q(\rho,s)$ is well-defined when $\rho\leq \varrho_m(s)$.} The flow is subsonic (i.e. $\mathfrak q(\rho,s)<c(\rho)$, where $c(\rho)$ is the sound speed defined in \eqref{eq:sound}) if and only if $\varrho_c(s)<\rho\leq \varrho_m(s)$. At the critical density one has $\mathfrak q(\varrho_c(s),s)= c(\varrho_c(s))$. We denote the square of the momentum and {the square of}  critical momentum by
\begin{equation}\label{tc}
	\t(\rho,s):=\rho^2 \mathfrak q(\rho,s)^2=2\rho^2(s-h(\rho)) \ \text{ and }\ \t_c(s):=\t(\varrho_c(s),s)=\left\{\frac{2(\gamma-1)}{\gamma+1} s\right\}^{\frac{\gamma+1}{\gamma-1}}.
\end{equation}

\begin{remark}\label{rmk:Q}
	Suppose that the flow $(\rho, \mathbf{u})$ has the mass flux $Q$ and  satisfies the asymptotics \eqref{asymptotic upstream} at the upstream  with a constant density $\bar \rho$. Then $\bar \rho$ must be defined by \eqref{eq:rhobar}. 
	The flow is subsonic at the upstream if $\tilde u^* < c(\bar \rho)=\bar\rho^{\frac{\gamma-1}{2}}$, where $\tilde u^*$ is defined in \eqref{ubar_extension}. This can be guaranteed by letting
	\begin{equation}\label{def:Q_*}
		Q>\tilde Q,\quad \text{where}\quad
		\tilde Q:=(\tilde u^*)^{\frac2{\gamma-1}}\int_0^{\bar H}y\bar u(y)\ dy\geq Q_*:=(\bar u^*)^{\frac2{\gamma-1}}\int_0^{\bar H}y\bar u(y)\ dy.
	\end{equation}
	
	An immediate consequence of \eqref{def:Q_*} is that the upper and lower bounds of the Bernoulli function $\mathcal{B}$ are comparable, i.e.
	\begin{equation}\label{def:t_bound}
		0<B_\ast \leq \mathcal{B}(z)\leq B^\ast \leq \frac{\gamma+1}{2}B_\ast,
	\end{equation}
	where $B_\ast$ and $B^\ast$ are defined in \eqref{defB*}. Furthermore, $B_\ast$ (thus $B^\ast$) is comparable to $\bar\rho^{\gamma-1}$, i.e.
	\begin{equation}\label{eq:rhobar_B_*}
		(\gamma-1)^{-\frac{1}{\gamma-1}}\bar \rho \leq B_\ast^{\frac{1}{\gamma-1}}\leq \left(\frac{\gamma+1}{2(\gamma-1)}\right)^{\frac{1}{\gamma-1}}\bar\rho.
	\end{equation}
	As $Q=\bar\rho \|y\bar u\|_{L^1([0,\bar{H}])}$, the above inequality can also be reformulated as
	\begin{align}\label{label_7}
		(\gamma-1)^{-\frac{1}{\gamma-1}}\frac{Q}{\|y\bar u\|_{L^1([0,\bar H])}}\leq B_\ast^{\frac{1}{\gamma-1}}\leq\left(\frac{\gamma+1}{2(\gamma-1)}\right)^{\frac{1}{\gamma-1}}\frac{Q}{\|y\bar u\|_{L^1([0,\bar H])}}.
	\end{align}
\end{remark}

Now we have the following lemma on the representation of the density $\rho$ in terms of the stream function $\psi$ in the subsonic region.

\begin{lemma}\label{g}
Suppose the density function $\r$ and the stream function $\psi$ satisfy the Bernoulli law \eqref{eq:psi}. Then the following statements hold.
\begin{itemize}
	\item [(i)] The density function $\r$ can be expressed as a function of $|\n\psi/y|^2$ and $\psi$ in the subsonic region, i.e.,
	\begin{equation}\label{rho g}
		\r=\frac 1{g(|\frac{\n\psi}{y}|^2,\psi)}, \q \text{if } \r\in (\varrho_c(\B(\psi)),\varrho_m(\B(\psi))],
	\end{equation}
	where $\varrho_c$ and $\varrho_m$ are functions defined in \eqref{defrhoc}, and 
	$$g:\{(t,z):0\leq t<\mathfrak t_c(\B(z)),\, z\in \R\}\to \R$$ 
	is a function smooth in $t$ and $C^{1,1}$ in $z$ with $\mathfrak t_c$ defined in \eqref{tc}. Furthermore, 
	\begin{equation}\label{g bound}
	\frac{1}{\varrho_m(B^*)}=:g_*\leq g(t,z)\leq g^*:= \frac{1}{\varrho_c(B_*)}.
	\end{equation}
	\item [(ii)] The function $g$ satisfies the identity 
	\begin{equation}\label{dzg dtg}
	g^2(t,z)\pt_z g(t,z)=-2\B'(z)\pt_t g(t,z), \quad t\in[0,\mathfrak t_c(\B(z))), \ z\in\R.
	\end{equation}
\end{itemize}
\end{lemma}

\begin{proof}
(i) From the expression \eqref{tc}, {the straightforward computations give}
\begin{align}\label{eq:dF}
	\pt_\rho \t(\rho,s)=4\rho\left(s-\frac{\gamma+1}{2}h(\rho)\right).
\end{align}
Now, one can see that with $\varrho_c$ and $\varrho_m$ defined in \eqref{defrhoc}, the following statements hold:
\begin{itemize}
	\item[(a)] $\rho\mapsto\t(\rho,s)$ achieves its maximum $\t_c(s)$ at $\rho=\varrho_{c}(s)$;
	%\item[(b)] $\mft(\rho,s)> 0$ if and only if $0< \rho< \varrho^\ast (s)$;
	\item[(b)] $\pt_\rho\t(\rho,s)<0$ when $\varrho_{c}(s)< \rho < \varrho_m(s)$.
\end{itemize}
Thus by the inverse function theorem, for each fixed $s>0$ and $\rho\in (\varrho_c(s), \varrho_m (s)]$, one can express $\rho$ as a function of $t:= \t(\rho,s)\in [0,\t_c(s))$ and $s$, i.e. $\rho=\rho(t,s)$. Let
\begin{equation} \label{eq:branch}
	g(t,z):=\frac{1}{\rho(t,\mathcal{B}(z))}, \quad t\in [0, \t_c(\B(z)))
\end{equation}
where $\mathcal{B}$ is the function defined in \eqref{defB}. The function $g$ is smooth in $t$ by the inverse function theorem and $C^{1,1}$ in $z$ by the $C^{1,1}$ regularity of $\mathcal{B}$ and the smooth dependence of $\rho$ on $s$. In view of the Bernoulli law \eqref{eq:psi}, one has \eqref{rho g} and \eqref{g bound}.

(ii) Let 
\begin{equation}\label{varrho_def}
	\varrho(t,z):=\rho(t,\mathcal{B}(z)).
\end{equation}
From \eqref{tc} and the Bernoulli law \eqref{eq:psi} one has
\begin{align}\label{eq:claim}
	t= 2\varrho(t,z)^2 (\mathcal{B}(z)-h(\varrho(t,z))).
\end{align}
Differentiating \eqref{eq:claim} with respect to $t$ gives
\begin{align}\label{eq:drhot}
	\frac{1}{2}=\frac{\pt_t\varrho}{\varrho}\left(2\varrho^2\mathcal{B}(z)-2\varrho^2h(\varrho)-\varrho^3h'(\varrho)\right)\stackrel{\eqref{eq:claim}}{=}\frac{\pt_t \varrho}{\varrho}\left(t-\varrho^{\gamma+1}\right),
\end{align}
and differentiating \eqref{eq:claim} with respect to $z$ gives
\begin{align}\label{eq40-1}
	-\varrho^2\mathcal{B}'(z)=\frac{\pt_z \varrho}{\varrho}\left(2\varrho^2\mathcal{B}(z)-2\varrho^2h(\varrho)-\varrho^3h'(\varrho)\right)\stackrel{\eqref{eq:claim}}{=}\frac{\pt_z\varrho}{\varrho}\left(t-\varrho^{\gamma+1}\right).
\end{align}
Combining  \eqref{eq:drhot} and \eqref{eq40-1} one has
\[
\pt_z \varrho(t,z)=-2\varrho(t,z)^2\pt_t\varrho(t,z) \mathcal{B}'(z).
\]
This together with \eqref{eq:branch} yields \eqref{dzg dtg}. 
\end{proof}

Now we are in a position to formulate the Euler system into a quasilinear equation of the stream function $\psi$ in the subsonic region.

\begin{lemma}\label{quasilinearequ lem}
Let $(\r,\u)$ be a solution to the system \eqref{eq:euler}.  Assume $(\r,\u)$ satisfies the assumptions in Proposition \ref{prop:Euler_equivalent}. Then in the subsonic region $|\n\psi/y|^2<\t_c(\B(\psi))$, the stream function $\psi$ solves
\begin{equation}\label{elliptic equ}
\n\c\bigg(g\bigg(\left|\frac{\n\psi}{y}\right|^2,\psi\bigg)\frac{\n\psi}{y}\bigg)
=\frac{y\B'(\psi)}{g(|\frac{\n\psi}{y}|^2,\psi)},
\end{equation}
where $g$ is defined in \eqref{eq:branch} and $\B$ is the Bernoulli function defined in \eqref{defB}. Moreover, the equation \eqref{elliptic equ} is elliptic if and only if $|\n\psi/y|^2<\t_c(\B(\psi))$.
\end{lemma}

\begin{proof}
	Let $\theta(s;X)$ be the streamlines satisfying
	\begin{equation*}
		\left\{
		\begin{aligned}
			& \frac{d\theta}{ds} =\mathbf{u}(\theta(s;X)),\\
			& \theta (0; X)=X.
		\end{aligned}
		\right.
	\end{equation*}
	It follows from the third equation in \eqref{eq:euler} that $\omega/(y\rho)$ is a constant along each streamline. Hence $\omega/(y\rho)$ can be determined by the associated data at the upstream as long as the streamlines of the flows have simple topological structure, which is guaranteed by the assumption $v<0$. If $X\in \{\psi=z\}$, then
	\begin{align*}
		\frac{\omega}{y\rho}(X)=\lim_{s \rightarrow -\infty}\frac{\omega}{y\rho}(\theta(s;X)) =-\frac{\bar u'(\h(z;\bar \rho))}{\h(z;\bar \rho)\bar\rho}.
	\end{align*}
	Expressing the vorticity $\omega$ in terms of the stream function $\psi$ and using \eqref{eq:dB} one has
	\begin{align*}
		-\nabla\cdot \left(\frac{\nabla \psi}{y\rho} \right)=\omega = -y\r\mathcal{B}'(\psi).
	\end{align*}
    In view of \eqref{rho g} the above equation can be rewritten into \eqref{elliptic equ}.

The equation \eqref{elliptic equ} can be written in the nondivergence form as follows
$$\mathfrak  a^{ij}\bigg(\frac{\n\psi}{y},\psi\bigg)\pt_{ij}\psi+\mathfrak b\bigg(\frac{\n\psi}{y},\psi\bigg)=\frac{y^2\B'(\psi)}{g(|\frac{\n\psi}{y}|^2,\psi)},$$
where the matrix
$$(\mathfrak a^{ij}):=g\bigg(\left|\frac{\n\psi}{y}\right|^2,\psi\bigg)I_{2}+2\pt_t g\bigg(\left|\frac{\n\psi}{y}\right|^2,\psi\bigg)\frac{\n\psi}{y}\otimes\frac{\n\psi}{y}$$
is symmetric with the eigenvalues
$$\b_0=g\bigg(\left|\frac{\n\psi}{y}\right|^2,\psi\bigg) \q{\rm and}\q \b_1=g\bigg(\left|\frac{\n\psi}{y}\right|^2,\psi\bigg)+2\pt_tg\bigg(\left|\frac{\n\psi}{y}\right|^2,\psi\bigg)\left|\frac{\n\psi}{y}\right|^2,$$
and
$$\mathfrak b=\pt_zg\bigg(\left|\frac{\n\psi}{y}\right|^2,\psi\bigg)|\nabla \psi|^2-\frac{\pt_y\psi}{y}\bigg[g\bigg(\left|\frac{\n\psi}{y}\right|^2,\psi\bigg)+2\pt_tg\bigg(\left|\frac{\n\psi}{y}\right|^2,\psi\bigg)\left|\frac{\n\psi}{y}\right|^2\bigg].
$$
Besides, differentiating the identity $t=\t(\frac{1}{g(t,z)},\B(z))$ gives
\begin{equation}\label{dtg drF}
\pt_tg(t,z)=-\frac{g^2(t,z)}{\pt_\rho\t(\frac1{g(t,z)},\B(z))} \q\text{for } t\in[0,\t_c(\B(z))).
\end{equation}
This implies
$$\pt_tg(t,z)\geq0 \q {\rm and}\q \lim_{t\to\t_c(\B(z))-}\pt_tg(t,z)=\infty.$$
Thus $\beta_0$ has uniform upper and lower bounds, and $\beta_1$ has a uniform lower bound but blows up when $|\nabla\psi/y|^2$ approaches $\t_c(\B(\psi))$. 
Therefore, the equation \eqref{elliptic equ} is elliptic as long as $|\n\psi/y|^2< \t_c(\B(\psi))$, and is singular when $|\n\psi/y|^2= \t_c(\B(\psi))$. This completes the proof of the lemma.
\end{proof}

\subsection{Reformulation for the jet flows in terms of the stream function}
Let 
\begin{equation}\label{Ld}
	\Lambda:=
	{\ul\rho\sqrt{2B(\bar H)-2h(\ul\rho)}} \q \text{with } \ul\rho :=(\gamma\ul p)^{1/\gamma}
\end{equation}
be the constant momentum on the free boundary, where {$B$ is defined in \eqref{def:B_upstream} and} $\ul p$ is the pressure on the free boundary as in Problem \ref{probelm 1}. From previous analysis, Problem \ref{probelm 1} can be solved as long as the following problem in terms of the stream function $\psi$ is solved.
\begin{problem}\label{pb2}
	Given a mass flux $Q>0$ and an axial velocity $\bar u\in C^{1,1}([0,\bar H])$ satisfying \eqref{ubar_condition}. One looks for a triple $(\psi, \Gamma_\psi, \Lambda)$ satisfying
	$\psi\in C^{2,\alpha}(\{\psi<Q\})\cap C^{1}(\overline{\{\psi<Q\}})$ for any $\alpha\in(0,1)$, $\partial_{x}\psi>0$ in $\{0<\psi<Q\}$ and
	\begin{equation}\label{eq:pb2}
		\left\{
		\begin{aligned}
			&\n\c\bigg(g\bigg(\left|\frac{\n\psi}{y}\right|^2,\psi\bigg)\frac{\n\psi}{y}\bigg)=\frac{y\B'(\psi)}{g(|\frac{\n\psi}{y}|^2,\psi)} &&\text{ in } \{0<\psi<Q\},\\
			&\psi =0 &&\text{ on } \N_0,\\
			&\psi =Q &&\text{ on } \N \cup \Gamma_\psi,\\
			&\left|\frac{\nabla \psi}y\right| =\Lambda &&\text{ on } \Gamma_\psi,
		\end{aligned}
		\right.
	\end{equation}
	where the free boundary $\Gamma_\psi:=\pt\{\psi<Q\}\backslash \N$.  Furthermore, the free boundary $\Gamma_\psi$ and the flow 
	\begin{align}\label{def:rhou_psi}
	(\r,\u)=\left(\frac1{g(|\n\psi/y|^2,\psi)},\frac{g(|\n\psi/y|^2,\psi)}{y}\pt_y\psi,-\frac{g(|\n\psi/y|^2,\psi)}{y}\pt_x\psi\right)
	\end{align}
	are expected to satisfy  the following properties.
	\begin{enumerate}
	
	\item The free boundary $\Gamma_\psi$ is given by a graph $x=\Upsilon(y)$, $y\in (\ubar{H}, 1]$ for some $C^{2,\alpha}$ function $\Upsilon$ and some $\ubar{H}\in (0,1)$.
	
	\item The free boundary $\G_\psi$ fits the nozzle at $A=(0,1)$ continuous differentiably, i.e., $\Upsilon(1)=N(1)$ and $\Upsilon'(1)=N'(1)$.
	
	\item For $x$ sufficiently large, the free boundary is also an $x$-graph, i.e., it can be written as $y=f(x)$ for some $C^{2,\alpha}$  function $f$. Furthermore, one has
	$$\lim_{x\to\infty}f(x)=\ul{H} \q \text{and}\q \lim_{x\to\infty}f'(x)=0.$$
	
	  \item The flow is subsonic in the flow region, i.e., $|\n\psi/y|^2<\t_c(\B(\psi))$ in $\{\psi<Q\}$, where $\t_c$ is defined in \eqref{tc} and the Bernoulli function $\B$ is defined in \eqref{defB}.
	
	\item  At the upstream the flow satisfies the asymptotic behavior \eqref{asymptotic upstream}, where the upstream density $\bar\rho$ is given by \eqref{eq:rhobar}. At the downstream the flow satisfies
	$$\lim_{x\to\infty}(\r,\u)=(\ul{\r},\ul{u},0)$$
	for some positive constant $\ul\rho$ and positive function $\ul u=\ul u(x_2)$.
	\end{enumerate}
\end{problem}

\begin{remark}\label{rmk_pb2}
	If we find a solution $\psi$ which solves Problem \ref{pb2}, then we obtain a jet flow $(\rho,\mathbf u)\in (C^{1,\alpha}(\{\psi<Q\})\cap C^{0}(\overline{\{\psi<Q\}}))^2$ by \eqref{def:rhou_psi}. The flow $(\rho,\mathbf u)$ is actually a solution of the Euler system \eqref{Euler system}, even if $\rho$ and $\mathbf u$ are not $C^{1,1}$ in the flow region as required in Proposition \ref{prop:Euler_equivalent}. For the proof we refer to \cite[Proposition 2.5]{LSTX2023} and \cite[Proposition 2.7]{L2023_axi_small}. 
\end{remark}

\subsection{Subsonic truncation}\label{subsec_subsonic}
The equation of the stream function $\psi$ in \eqref{eq:pb2} becomes degenerate when the flows approach the sonic state (cf. Lemma \ref{quasilinearequ lem}). In order to deal with this possible degeneracy, we introduce the following subsonic truncation. 

Let $\varpi:\R\to[0,1]$ be a smooth nonincreasing function such that
$$\varpi(s)=\begin{cases}
	1& \text{ if } s\leq-1,\\
	0& \text{ if } s\leq -\frac12,
\end{cases}
\q
{\rm and}\q
|\varpi'|+|\varpi''|\leq8.$$
For $\epsilon\in (0,1/4)$,
let $\varpi_\epsilon(s):=\varpi((s-1)/\epsilon)$. We define $g_\epsilon:[0,\infty)\times \R\rightarrow \R$,
\begin{equation}\label{eq:truncation_g}
	g_\epsilon(t,z):= g(t, z)\varpi_\epsilon(t/\t_c(\B(z)))+ (1-\varpi_\epsilon(t/\t_c(\B(z))))g^* ,
\end{equation}
where $\t_c$ is defined in \eqref{tc} and $g^*$ is the upper bound for $g$ in \eqref{g bound}. The properties of $g_\epsilon$ are summarized in the following lemma.

\begin{lemma}\label{lem:truncation_g}
Let $g$ be the function defined in Lemma \ref{g} and let $g_\epsilon$ be the subsonic truncation of $g$ defined in \eqref{eq:truncation_g}. Then the function $g_\epsilon(t, z)$ is smooth with respect to $t$ and {$C^{1,1}$} with respect to $z$. Furthermore, under the assumption \eqref{def:Q_*}, there exist  positive constants $c_\ast$ and $c^\ast$ depending only on $\gamma$, such that for all $(t,z)\in [0,\infty)\times \R$
\begin{align}
	&c_\ast B_*^{-\frac{1}{\gamma-1}} \leq g_\epsilon(t,z)\leq c^\ast B_*^{-\frac{1}{\gamma-1}},\,\label{eq:g}\\
	&c_\ast B_*^{-\frac{1}{\gamma-1}}\leq g_\epsilon(t, z)+2\partial_t g_\epsilon(t, z) t\leq c^\ast \epsilon^{-1}B_*^{-\frac{1}{\gamma-1}},\label{beta1eps}\\
	&|\pt_zg_\epsilon(t,z)|\leq c^\ast\kappa_0\epsilon^{-1}B_\ast^{-\frac{\gamma+1}{\gamma-1}},\quad  |t\pt_z g_\epsilon(t,z)|\leq c^\ast \kappa_0\epsilon^{-1}.
	\label{eq:dzg}
\end{align}
\end{lemma}
\begin{proof}
	The proof is the same as that in \cite[Lemma 2.6]{LSTX2023}. For later use we give the details here. 
	
    \emph{(i)}. It follows from Lemma \ref{g} and the definition of $\varpi_\epsilon$ that $g_\epsilon$ is smooth with respect to $t$ and {$C^{1,1}$} with respect to $z$. Clearly, \eqref{eq:g} follows directly from the upper and lower bound for $g$ in \eqref{g bound} and the definition of $g_\epsilon$ in \eqref{eq:truncation_g}.
	
	\emph{(ii)}. To show \eqref{beta1eps},  we first claim that if $0\leq t\leq\left(1-\frac{\epsilon}2\right) \t_c(\B(z))$ then \begin{equation}\label{eq:g_dt_bound}
		0<\pt_t g(t,z)\leq \frac{C_\gamma  g(t,z)}{\epsilon \t_c(B_\ast)},
	\end{equation}
	where $C_\gamma$ represents a positive constant depending only on $\gamma$. In view of the expression of $\pt_t g(t,z)$ in \eqref{dtg drF} it suffices to estimate $\pt_{\rho}\t(\rho,s)$ at $\rho=\rho(t,s)$ for $0\leq t\leq\left(1-\frac{\epsilon}2\right) \t_c(s)$. In fact, it follows from \eqref{defrhoc} that one has  $s=\frac{\gamma+1}{2}h(\varrho_c(s))$. Hence for $0\leq t\leq\left(1-\frac{\epsilon}2\right) \t_c(s)$  and $\rho\geq \varrho_c(s)$ it holds that
	\begin{align*}
		\partial_{\rho}\t(\rho, s)&\stackrel{\eqref{eq:dF}}{=}2(\gamma+1)\rho\left(h(\varrho_c(s))-h(\rho)\right)\stackrel{\eqref{tc}}{=}2(\gamma+1)\rho\left(\frac{
			\t(\rho,s)}{2\rho^2}-\frac{\t_c(s)}{2\varrho_c(s)^2}\right)\\
		&\leq (\gamma+1)\frac{\t(\rho,s)-\t_c(s)}{\rho}\leq -\frac{(\gamma+1)\epsilon \t_c(s)}{2\rho}.
	\end{align*}
	Thus from \eqref{dtg drF} and \eqref{def:t_bound} we conclude that
	\begin{equation*}
		0<\pt_tg(t,z)=-\frac{g^2(t,z)}{\pt_\rho\t(\rho, s)}\Big|_{\rho=\frac{1}{g(t,z)}, s=\mathcal{B}(z)}\leq \frac{2g(t,z)}{(\gamma+1)\epsilon\t_c(\mathcal{B}(z))}\leq \frac{C_\gamma  g(t,z)}{\epsilon \t_c(B_\ast)},
	\end{equation*}
	which is the claimed inequality \eqref{eq:g_dt_bound}.
	With \eqref{eq:g_dt_bound} at hand, we have
	\begin{equation}\label{eq:g_dt}
		\begin{split}
			\pt_tg_\epsilon(t,z) &= \partial_t g (t, z) \varpi_\epsilon\left(\frac{t}{\t_c(\B(z))}\right)+(g(t, z)-g^*) \varpi'_\epsilon\left(\frac{t}{\t_c(\B(z))}\right)\frac{1}{\t_c(\B(z))}\\
			&\leq \frac{C_\gamma  g(t,z)}{\epsilon \t_c(B_\ast)}+\frac{8g^\ast}{\epsilon \t_c(B_\ast)} \leq \frac{C_\gamma g^\ast}{\epsilon \t_c(B_\ast)}.
		\end{split}
	\end{equation}
	Note that 
	\begin{equation}\label{dtgm_0}
		\pt_tg_\epsilon(t,z)=0 \q\text{in }\{(t,z): t\geq (1-\epsilon/2)\t_c(\mathcal{B}(z)), z\in \R\}.
	\end{equation} 
	Thus
	\begin{align*}
		0\leq t\pt_tg_\epsilon(t,z)\leq \frac{C_\gamma g^\ast \t_c(\B(z))}{\epsilon \t_c(B_\ast)}\leq \frac{C_\gamma g^\ast \t_c(B^\ast)}{\epsilon \t_c(B_\ast)}.
	\end{align*}
	Since $B_\ast$ and $B^\ast$ are equivalent, cf. \eqref{def:t_bound}, then the estimate \eqref{beta1eps} follows directly from  \eqref{eq:g}.
	
	\emph{(iii).}  Direct computations give
		\begin{equation}\label{gm_dz}
			\begin{aligned}
				&\pt_z g_\epsilon (t,z)= \pt_z g(t,z)\varpi_\epsilon\left(\frac{t}{\t_c(\B(z))}\right) + (g^\ast-g(t,z))\varpi_\epsilon'\left(\frac{t}{\t_c(\B(z))}\right) \frac{t\t_c'(\B(z))\B'(z)}{\t_c(\B(z))^2}\\
				\stackrel{\eqref{dzg dtg}}{=} & \left\{ -2\frac{\pt_t g(t,z)}{ g(t,z)^2} \varpi_\epsilon\left(\frac{t}{\t_c(\B(z))}\right)+ (g^\ast-g(t,z))\varpi_\epsilon'\left(\frac{t}{\t_c(\B(z))}\right) \frac{t\t_c'(\B(z))}{\t_c(\B(z))^2}\right\}\B'(z).
			\end{aligned}
		\end{equation}
		Note that both sums in the expression of $\pt_tg_\epsilon$, cf. \eqref{eq:g_dt}, are nonnegative. Thus we get from  the expression of $\pt_zg_\epsilon$ in \eqref{gm_dz} that
		\begin{equation}\label{eq:dzgm_dtgm}\begin{split}
				|\pt_zg_\epsilon(t,z)|\leq& \max\left\{\frac{2}{g^2(t,z)},\frac{t\mathfrak t_c'(\mathcal B(z))}{\mathfrak t_c(\mathcal B(z))}\right\}|\mathcal B'(z)| \pt_tg_\epsilon(t,z).
			\end{split}
		\end{equation}
		By \eqref{eq:u0_eps0_B} and \eqref{eq:rhobar_B_*} one has 
		\begin{equation}\label{eq:Bdz}
			|\mathcal{B}'(z)|%\leq \frac{\kappa_0}{\bar \rho}
			\leq C_\gamma\kappa_0B_\ast^{-\frac{1}{\gamma-1}}.
		\end{equation}
		Combining the above two estimates with  \eqref{g bound}, 
		\begin{align}\label{lable_2}
			\t_c(\B(z))\sim B_\ast^{\frac{\gamma+1}{\gamma-1}},
			\quad \t'_c(\B(z))\sim B_\ast^{\frac{2}{\gamma-1}}
		\end{align}
		and \eqref{eq:g_dt}-\eqref{dtgm_0} yields the first inequality in \eqref{eq:dzg}.
		
		To show the second inequality in \eqref{eq:dzg}, first we note that from \eqref{eq:dzgm_dtgm} and \eqref{dtgm_0} one has
		\begin{equation}\label{dzgm_0}
			\pt_zg_\epsilon(t,z)=0 \q\text{in }\{(t,z): t\geq (1-\epsilon/2)\t_c(\mathcal{B}(z)), z\in \R\}.
		\end{equation} 
		Then it follows from the first inequality in \eqref{eq:dzg} that
		%\begin{align*}
		%\left|t\p_zg_\epsilon(t,z)\right|\leq C_\gamma \left|\frac{t\p_tg(t,z)}{g(t,z)^2} + \frac{g^\ast \mft'_c(\mathcal{B}(z))}{\epsilon}\right||\mathcal{B}'(z)|\leq \frac{C_\gamma B_\ast^{\frac{1}{\gamma-1}}}{\epsilon}|\mathcal{B}'(z)|.
		%\end{align*}
		\begin{align*}
			\left|t\pt_zg_\epsilon(t,z)\right|\leq \t_c(\mathcal{B}(z))|\pt_zg_\epsilon(t,z)|\leq C_\gamma\kappa_0\epsilon^{-1}.
		\end{align*}
		This finishes the proof of the lemma.
\end{proof}

\section{Variational formulation for the free boundary problem}\label{sec variational formulation}

As in \cite[Section 3]{L2023_axi_small}, we can show that the quasilinear equation of the stream function is an Euler-Lagrange equation for an energy functional. Thus the jet problem can be transformed into a variational problem.

In Sections \ref{sec variational formulation}-\ref{sec fine properties}, we will always assume that the mass flux $Q$ satisfies  $Q>\tilde Q$, where $\tilde Q$ is defined in \eqref{def:Q_*}.

Let $\O$ be the domain bounded by $\N_0$ (the $x$-axis) and $\N\cup([0,\infty)\times\{1\})$. Since $\Omega$ is unbounded, we make an  approximation by considering the problems in a family of truncated domains $\O_{\mu,R}:=\O\cap\{-\mu<x<R\}$, where $\mu$ and $R$ are positive numbers. Set
\begin{equation}\label{G def}
G_\v(t,z):=\frac12\int_0^tg_\v(s,z)ds+\frac{1}{\g}(g^{-\g}_\v(0,z)-g_\v^{-\g}(0,Q)).
\end{equation}
and
\begin{equation}\label{Phi def}
\Phi_\v(t,z):=-G_\v(t,z)+2\pt_tG_\v(t,z)t.
\end{equation}
Given  $\psi^\sharp_{\mu,R}\in C(\pt\Omega_{\mu,R})\cap H^1(\Omega_{\mu,R})$ with $0\leq \psi^\sharp_{\mu,R}\leq Q$, we consider the following minimization problem:
\begin{equation}\label{variation problem}
	\text{find }\psi\in K_{\psi^\sharp_{\mu,R}}  \, \text{ s.t. } \, J_{\mu,R,\Ld}^\v(\psi)= \inf_{\phi\in K_{\psi^\sharp_{\mu,R}}} J_{\mu,R,\Ld}^\v(\phi),
\end{equation}
where 
\begin{equation}\label{Jm def}
	J^\v_{\mu,R,\Ld}(\phi):=\int_{\Omega_{\mu,R}}y\bigg[G_\v\bigg(\left|\frac{\n\phi}{y}\right|^2,\phi\bigg)+\ld_\v^2\chi_{\{\phi<Q\}}\bigg]dX, \q\q \ld_\v: =\sqrt{\Phi_\v(\Ld^2,Q)}
\end{equation}
with $\Ld$ given by \eqref{Ld}, and where
$$K_{\psi^\sharp_{\mu,R}}:=\{\phi\in H^1(\O_{\mu,R}):\phi=\psi^\sharp_{\mu,R}\text{ on } \pt\O_{\mu,R}\}.$$
Note that $\Phi_\epsilon(\Lambda^2,Q)>0$.
Indeed, with the aid of ellipticity condition \eqref{beta1eps},  straightforward computations show that $t\mapsto \Phi_\epsilon(t,z)$ is monotone increasing:
\begin{align}\label{Phim dt}
	0<\frac12	c_\ast B_*^{-\frac{1}{\gamma-1}}\leq\pt_t\Phi_\v(t,z)=\frac12g_\v(t,z)+\pt_tg_\v(t,z)t\leq\frac12	c^\ast\v^{-1} B_*^{-\frac{1}{\gamma-1}}.
\end{align}
Since $\Phi_\epsilon(0,Q)=-G_\epsilon(0,Q)=0$, one immediately has $\Phi_\epsilon(t,Q)>0$ for $t>0$.

The existence and the H\"{o}lder regularity of minimizers for \eqref{variation problem} have been shown in \cite[Lemmas 4.2 and 4.7]{L2023_axi_small} (see also Lemmas \ref{minimizer existence} and \ref{lemmas}). We now derive the Euler-Lagrange equation for the variational problem \eqref{variation problem} in the open set $\O_{\mu,R}\cap\{\psi<Q\}$.
	
\begin{lemma}\label{EL}
Let $\psi$ be a minimizer of problem \eqref{variation problem}. Then $\psi$ is a solution to
\begin{equation}\label{EL equ}
\n\c\bigg(g_\v\bigg(\left|\frac{\n\psi}{y}\right|^2,\psi\bigg)\frac{\n\psi}{y}\bigg)=y\pt_zG_\v\bigg(\left|\frac{\n\psi}{y}\right|^2,\psi\bigg) \q\text{in } \O_{\mu,R}\cap\{\psi<Q\}.
\end{equation}
Furthermore, if $|\n\psi/y|^2\leq(1-\v)\t_c(\B(\psi))$, it holds that
\begin{equation}\label{Gdz expression}
\pt_zG_\v\bigg(\left|\frac{\n\psi}{y}\right|^2,\psi\bigg)=\frac{\B'(\psi)}{g(|\frac{\n\psi}{y}|^2,\psi)}.
\end{equation}
\end{lemma}
\begin{proof}
%The proof is basically the same as that in \cite[Lemma 3.1]{LSTX2023} and \cite[Lemma 3.1]{L2023_axi_small}. 	 
For any $\eta\in C^\infty_0(\O_{\mu,R}\cap\{\psi<Q\})$,  direct computations give
$$\frac{d}{d\tau}J_{\mu,R,\Ld}^\v(\psi+\tau\eta)\bigg|_{\tau=0}
=\int_{\O_{\mu,R}}2\pt_tG_\v\bigg(\left|\frac{\n\psi}{y}\right|^2,\psi\bigg)\frac{\n\psi}{y}\c\n\eta
+y\pt_zG_\v\bigg(\left|\frac{\n\psi}{y}\right|^2,\psi\bigg)\eta.$$
By the definition of $G_\v$ in \eqref{G def}, minimizers of \eqref{variation problem} satisfy the equation \eqref{EL equ}.

Next, by the definition of $G_\epsilon$ in \eqref{G def}, one has
$$\pt_z G_\epsilon(t,z)=\frac{1}{2}\int_0^t \pt_z g_\epsilon(\tau, z) d\tau-g^{-\gamma-1}_\epsilon(0,z)\pt_zg_\epsilon(0,z).$$
Note that 
\begin{equation}\label{gm_g}
	g_\epsilon(t,z)=g(t,z) \q\text{in } \mathcal T_\epsilon:=\{(t,z): 0\leq t\leq(1-\epsilon)\t_c(\B(z)), z\in \R\}.
\end{equation} 
Thus using \eqref{dzg dtg} one has
\begin{align*}
	\pt_zG_\epsilon(t,z)=\mathcal{B}'(z)\varrho(t,z)-\mathcal{B}'(z)\varrho(0,z)-g^{-\gamma-1}(0,z)\pt_zg(0,z)\quad\text{in } \mathcal T_\epsilon.
\end{align*}
It follows from \eqref{eq:drhot}  that $\pt_t\varrho (0,z)=-\frac{1}{2}\varrho^{-\gamma}(0,z)$. This together with \eqref{dzg dtg} gives
\begin{align}
	\label{eq:dg_0}
	\pt_zg(0,z)=-\mathcal{B}'(z) \varrho^{-\gamma}(0,z)=-\mathcal{B}'(z)g^{\gamma}(0,z) .
\end{align}
Substituting  \eqref{eq:dg_0} into the expression of $\pt_zG_\epsilon$ one gets
\begin{align}\label{eq:dzG}
	\pt_zG_\epsilon(t,z)=\mathcal{B}'(z)\varrho(t,z)=\frac{\mathcal{B}'(z)}{g(t,z)} \quad \text{in } \mathcal T_\epsilon.
\end{align}
This completes the proof for the lemma.
\end{proof}

Minimizers of problem \eqref{variation problem} satisfy the following free boundary condition. 

\begin{lemma}(\cite[Lemma 3.2]{L2023_axi_small})\label{free BC}
Let $\psi$ be a minimizer of problem \eqref{variation problem}. Then
$$\Phi_\v\bigg(\left|\frac{\n\psi}{y}\right|^2,\psi\bigg)=\ld_\v^2 \quad\text{on } \G_{\psi}:=\pt\{\psi<Q\}\cap\O_{\mu,R}$$
in the sense that
$$\lim_{s\to0+}\int_{\pt\{\psi<Q-s\}}y\bigg[\Phi_\v\bigg(\left|\frac{\n\psi}{y}\right|^2,\psi\bigg)-\ld_\v^2\bigg](\eta\c\nu)d\mathcal{H}^1=0 \q\text{for any }
\eta\in C_0^\infty(\O_{\mu,R};\R^2),$$
where $\mathcal{H}^1$ is the one-dimensional Hausdorff measure.
\end{lemma}

\begin{remark}[Relation between $\lambda_\epsilon$ and $\Lambda$]\label{rmk:ld Ld}
	If the free boundary $\Gamma_\psi$ is smooth and $\psi$ is smooth near $\Gamma_\psi$, then it follows from the monotonicity of $t \mapsto \Phi_\epsilon(t,z)$ that for each $z$ and the definition of $\lambda_\epsilon$ in \eqref{Jm def}  that $|\nabla\psi/y|=\Lambda$ on $\Gamma_\psi.$
		Moreover, since
		$$\lambda_\epsilon^2=\Phi_\epsilon(\Lambda^2, Q) =\Phi_\epsilon(0, Q)+\int_0^{\Lambda^2}\pt_t\Phi_\epsilon(s,Q)ds,$$
		using \eqref{Phim dt} and $\Phi_\epsilon(0,Q)=0$ one has
		\begin{align*}
			\frac{1}{2}c_\ast\Lambda^2	\leq B_\ast^{\frac{1}{\gamma-1}}\lambda_\epsilon^2\leq \frac12c^\ast \epsilon^{-1}\Lambda^2.
		\end{align*}
		In view of \eqref{eq:rhobar_B_*} and \eqref{eq:rhobar} one concludes that there exists a constant $C>0$ depending on $\gamma,\,\epsilon$, and $\|y\bar u\|_{L^1([0,\bar H])}$ such that
		\begin{align}\label{ld_Ld}
			C^{-1}\leq \frac{Q^{\frac{1}{2}}\lambda_\epsilon}{\Lambda}\leq C.
		\end{align}
\end{remark}

\section{The existence and regularity for the free boundary problem}\label{sec existence and regularity}

In this section, we study the existence and the regularity of minimizers for the minimization problem \eqref{variation problem}, as well as the regularity of the free boundary away from the nozzle. 
%The idea is the same as that in \cite[Section 4]{L2023_axi_small}. 

\subsection{Existence of minimizers}\label{subsec_existence}
For the ease of notations in the rest of this section, let
\begin{equation}\label{notation JDG}
\D:=\O_{\mu,R}=\O\cap\{-\mu<x<R\},\q
\MG(\bp,z):=G_\v(|\bp|^2,z),\q  \J:=J^\v_{{\mu,R},\Ld}, \q \ld:=\ld_\v.
\end{equation}
Then $\D$ is a bounded Lipschitz domain in $\R^2$ and is contained in the infinite strip $\mathbb R\times[0,\bar H]$, $\MG: \R^2\times \R\to \R$ is smooth in $\bp$ and $C^{1,1}$ in $z$, and $\ld$ is a positive constant. The minimization problem \eqref{variation problem} can be rewritten as
\begin{equation}\label{minimization problem}
\text{find }\psi\in \K_{\psi^\sharp}  \, \text{ s.t. } \, \J(\psi)= \inf_{\phi\in \K_{\psi^\sharp}} \J(\phi)
\end{equation}
with
\begin{equation*}\label{J}
\J(\phi):=\int_\D y\left[\MG\left(\frac{\n\phi}{y},\phi\right)+\ld^2\chi_{\{\phi<Q\}}\right]dX
\end{equation*}
and 
$$\K_{\psi^\sharp}:=\{\phi\in H^1(\D):\phi=\psi^\sharp \text{ on } \pt\D\}.$$
Here $\psi^\sharp$ is a given continuous function satisfying $0\leq \psi^\sharp\leq Q$ on $\pt\D$.

The properties of $\MG$ are summarized in the following proposition.

\begin{proposition}\label{Gproperties pro}
Let $G_\epsilon$ be defined in \eqref{G def} and $\MG$ be defined in \eqref{notation JDG}, then the following properties hold.
\begin{enumerate}
	\item[(i)] {There exist  positive constants  $\mathfrak b_\ast=c_\ast B_\ast^{-\frac{1}{\gamma-1}}$ and $\mathfrak b^\ast=c^\ast B_\ast^{-\frac{1}{\gamma-1}}$ with $c_\ast$ and $c^\ast$ depending only on $\gamma$,  such that}
	\begin{align}
		\mathfrak b_*|\bp|^2 &\leq p_i \partial_{p_i} \MG (\bp,z) \leq \mathfrak b^*|\bp|^2,\label{eq:convex}\\
		\mathfrak b_* |\mathbf\xi |^2 &\leq \xi_i\partial_{p_ip_j} \MG(\bp,z)\xi_j \leq \mathfrak b^* \epsilon^{-1} |\mathbf\xi|^2 \quad\text{for all }\mathbf\xi\in \R^2.\label{eq:convex0}
	\end{align}
	\item[(ii)] One has
	\begin{equation}\label{supportG}
	\begin{split}
	&\pt_z\MG(\bp,z)= 0 \quad\text{in } \{(\bp, z): \bp\in\R^2, z\in (-\infty,0]\},\\
   &\pt_z\MG(\bp,z)\geq 0 \quad \text{in } \{(\bp, z): \bp\in\R^2, z\in (0,\infty)\}.
	\end{split}
	\end{equation}
	\item[(iii)] There exist constants
	\begin{equation}\label{delta}
		\delta':=\epsilon^{-1}C_\gamma \kappa_0  \quad\text{and}\q
		%\delta:= \delta' B_\ast^{-\frac{1}{\gamma-1}}\left(B_\ast ^{-\frac{\gamma-1}{2(\gamma+1)}}+\kappa_0 B_\ast^{-1}+\bar u_\ast^{-1}\right),
		\delta:= \delta' B_\ast^{-\frac{1}{\gamma-1}}\left(
		B_\ast^{-\frac12} +\kappa_0 B_\ast^{-1}+\bar u_\ast^{-1}\right),	
	\end{equation}
	where $\kappa_0$ is defined in \eqref{k0} and $C_\gamma>0$ is a constant depending only on $\g$, such that
	\begin{align}
		&\epsilon^{-1}|\pt_z \MG(\bp,z)|+|\bp\cdot \pt_{\bp z} \MG(\bp,z)|\leq \delta', \quad |\pt_{\bp z}\MG(\bp,z)|+|\pt_{zz}\MG(\bp,z)|\leq \delta, \label{eq:upper_pzzG}\\
		&\MG(\mathbf 0,Q)=0,\quad \MG(\bp,z)\geq {\frac{\mathfrak b_*}2}|\bp|^2-
		C_\gamma\kappa_0%\delta'\epsilon
		{\min\{Q, (Q-z)_+\}}.\label{eq:com_energy}
	\end{align}
\item[(iv)] One has 
\begin{align}
	&\pt_{zz}\MG(\bp,z)\geq0 
	\q\text{and}\q
	{\rm det}\begin{bmatrix}
			\pt_{\bp\bp}\MG(\bp,z)& y\pt_{\bp z}\MG(\bp,z)\\
			\pt_{\bp z}\MG(\bp,z) &	y\pt_{zz}\MG(\bp,z)
		\end{bmatrix}
		\geq0  \label{det_0}
	\end{align}
in the set $\mathcal P_\v:=\{(\bp,z): |\bp|^2\leq\left(1-{\v}\right)\t_c(\B(z)),z\in\mathbb R\}$.
\end{enumerate}
\end{proposition}
\begin{proof}
The proof for (i) and (iii) is the same as that in \cite[Proposition 4.1]{LSTX2023}. 

As for (ii), note that $\pt_z \MG(\bp, z)$ is a positive multiple of $\mathcal{B}'(z)$ (cf. the proof for \cite[Proposition 4.1(ii)]{LSTX2023}). Thus in view of \eqref{eq:u0_eps0_B} and \eqref{eq:sign_B} one gets \eqref{supportG}. 

To prove (iv), notice by \eqref{eq:dzG} 
\begin{equation*}\label{Gdz}
	\pt_z\MG(|\bp|^2,z)=\frac{\B'(z)}{g(|\bp|^2,z)} \q\text{in }\mathcal P_\v.
\end{equation*}
Hence
\begin{equation}\label{dzzG}\begin{split}
	\pt_{zz}\MG(|\bp|^2,z)=&\frac{\B''(z)}{g(|\bp|^2,z)}-\frac{\B'(z)}{g^2(|\bp|^2,z)}\pt_zg(|\bp|^2,z)\\
	\stackrel{\eqref{dzg dtg}}{=}&\frac{\B''(z)}{g(|\bp|^2,z)}+\frac{2(\B'(z))^2}{g^4(|\bp|^2,z)}\pt_tg(|\bp|^2,z)
\end{split}\end{equation}
in $\mathcal P_\v$. 
Then the first inequality in \eqref{det_0} follows from $\B''(z)\geq0$ (cf. \eqref{eq:u0_eps0_B}) and $\pt_tg(t,z)>0$ (cf. \eqref{eq:g_dt_bound}). 
Furthermore, using the definition of $G_\epsilon$ in \eqref{G def},  straightforward computations yield 
\begin{equation*}
	\begin{split}
		&\pt_{p_ip_j} \MG(\bp,z) = g(|\bp|^2, z)\delta_{ij} + 2\pt_t g(|\bp|^2,z) p_ip_j \quad\text{and}
		\quad \pt_{p_i z}\MG(\bp,z) = p_i \pt_zg(|\bp|^2,z)
	\end{split}
\end{equation*}	
in $\mathcal P_\v$. 
These together with the expression of $\pt_{zz}\MG$ in \eqref{dzzG} and the relation between $\pt_zg$ and $\pt_tg$ in \eqref{dzg dtg} give the second inequality of \eqref{det_0}. This finishes the proof. 
\end{proof}

The existence of minimizers for the minimization problem \eqref{minimization problem} follows from standard theory for calculus of variations.

\begin{lemma}(\cite[Lemma 4.2]{L2023_axi_small})\label{minimizer existence} 
Assume $\MG$ satisfies \eqref{eq:convex0} and \eqref{eq:com_energy}. Then the minimization problem \eqref{minimization problem} has a minimizer.
\end{lemma}

\subsection{Lipschitz regularity and nondegeneracy of minimizers}
With Proposition \ref{Gproperties pro} at hand, we can establish the (optimal) Lipschitz regularity and nondegeneracy property of minimizers as in \cite[Section 4]{L2023_axi_small}.

First, in view of \cite[Section 4.2]{L2023_axi_small}, minimizers of  \eqref{minimization problem} have the following properties. 

\begin{lemma}\label{lemmas}
Let $\psi$ be a minimizer of \eqref{minimization problem}.
Then the following statements hold. 
\begin{itemize}
	\item [(i)] The function $\psi$ is a supersolution of the elliptic equation
	\begin{equation}\label{elliptic equG}
		\pt_i\pt_{p_i}\MG\left(\frac{\n\psi}{y},\psi\right)-y\pt_z\MG\left(\frac{\n\psi}{y},\psi\right)=0,
	\end{equation}
	in the sense of 
	\begin{equation}\label{supersolution}
		\int_{\D}\left[\pt_{p_i}\MG\left(\frac{\n\psi}{y},\psi\right)\pt_i\eta+y\pt_z\MG\left(\frac{\n\psi}{y},\psi\right)\eta\right]\geq0,
		\q\text{for all } \eta\geq0, \ \eta\in C_0^\infty(\D).
	\end{equation}
	
	\item[(ii)] The function $\psi$ satisfies $0\leq \psi\leq Q$ in $\D$.  
	
	\item[(iii)] The function $\psi\in C_{{\rm loc}}^{0,\a}(\D)$ for any $\a\in(0,1)$.
\end{itemize}
\end{lemma}

We consider the rescaled and renormalized function
\begin{equation}\label{def:psi*}
	\psi^*_{\bar X,r}(X):=\frac{Q-\psi(\bar X+rX)}{Qr},\q \text{for} \  \bar X\in\D, \ r\in(0,1).
\end{equation}
Then $\psi^*_{\bar X,r}$ is a minimizer of
\begin{equation*}\label{J scale}
	\J_{\bar X,r}(\phi):=\int_{\D_{\bar X,r}}(\bar y+ry)\left[\tilde\MG_r\left(\frac{\n\phi}{\bar y+ry},\phi\right)
	+\ld^2\chi_{\{\phi>0\}}\right]dX
\end{equation*}
over the admissible set $\K_{\bar X,r,\psi^{\sharp}}:=\left\{\phi\in H^1(\D_{\bar X,r}):\phi=\psi^*_{\bar X,r} \text{ on } \pt\D_{\bar X,r}\right\}$, where
\begin{equation}\label{Gsacle}
	\tilde\MG_r(\bp,z):=\MG(-Q\bp,Q-Qrz) \q\text{and}\q
	\D_{\bar X,r}:=\left\{(X-\bar X)/r: X\in\D\right\}.
\end{equation}
The straightforward computations show that $\psi^*_{\bar X,r}$ satisfies 
\begin{equation}\label{ellipticsub scale}
	\pt_i\pt_{p_i}\tilde\MG_r\left(\frac{\n\psi^*_{\bar X,r}}{\bar y+ry},\psi^*_{\bar X,r}\right)-(\bar y+ry)\pt_z\tilde\MG_r\left(\frac{\n\psi^*_{\bar X,r}}{\bar y+ry},\psi^*_{\bar X,r}\right)\geq0
\end{equation}
in $\D_{\bar X,r}$, and
\begin{equation}\label{ellipticequ scale}
	\pt_i\pt_{p_i}\tilde\MG_r\left(\frac{\n\psi^*_{\bar X,r}}{\bar y+ry},\psi^*_{\bar X,r}\right)-(\bar y+ry)\pt_z\tilde\MG_r\left(\frac{\n\psi^*_{\bar X,r}}{\bar y+ry},\psi^*_{\bar X,r}\right)=0
\end{equation}
in $\D_{\bar X,r}\cap\{\psi^*_{\bar X,r}>0\}$. 

The properties of $\MG$ in Proposition \ref{Gproperties pro} can be translated into the properties of $\tilde\MG_{r}$ in an obvious fashion:
\begin{align}
	&\tilde{\mathfrak b}_\ast|\bp|^2 \leq p_i \pt_{p_i} \tilde{\MG}_r (\bp, z)\leq \tilde{\mathfrak b}^\ast |\bp|^2,\label{eq:Gm_pdp}\\
	&\tilde{\mathfrak b}_\ast |\xi|^2 \leq \xi_i \pt_{p_ip_j}\tilde{\MG}_r(\bp,z)\xi_j\leq \tilde{\mathfrak b}^\ast \epsilon^{-1}|\xi|^2 \quad\text{for all }\mathbf\xi\in \R^2, \label{eq:Gm_dpp}\\
	&\pt_z\tilde\MG_r(\bp, z)\leq 0\quad \text{in } \{(\bp, z): \bp\in\R^2, z\in \mathbb R\},\label{eq:Gm_dz0}\\
	&\epsilon^{-1}|\pt_z \tilde\MG_{r}(\bp,z)|+|\bp\cdot \pt_{\bp z} \tilde{\MG}_r(\bp,z)|\leq \tilde{\delta'} r,\q
	|r\pt_{\bp z}\tilde\MG_{r}(\bp,z)|+|\pt_{zz}\tilde\MG_{r}(\bp,z)|\leq \tilde{\delta}r^2,\label{eq_rescale0}\\
	&\tilde\MG_r(\mathbf{0},0)=0,\q \tilde{\MG}_r(\bp,z)\geq\frac{\tilde{\mathfrak b}_*}{2}|\bp|^2-\tilde\delta'\v r z_+, \label{eq_rescale2}
\end{align}
where $\tilde{\mathfrak b}_\ast := Q^2 \mathfrak b_\ast$, $\tilde{\mathfrak b}^\ast:= Q^2\mathfrak b^\ast$, $\tilde{\delta'}:=\delta' Q$,  $\tilde{\delta}:=\delta Q^2$ with $\mathfrak b_\ast$, $\mathfrak b^\ast$,  $\delta'$, $\delta$ as in Proposition \ref{Gproperties pro}.  Note that from the explicit expressions of $\mathfrak b_\ast$, $\mathfrak b^\ast$, $\d'$, and $\delta$ in Proposition \ref{Gproperties pro} and Remark \ref{rmk:Q} one has 
\begin{align}\label{eq:BQ2}
	\tilde{\mathfrak b}_\ast, \tilde{\mathfrak b}^\ast, \tilde\delta',\tilde\delta \sim Q,
\end{align}
where $A\sim B$ means that $C^{-1}B\leq A\leq CB$ for some $C=C(\gamma,\epsilon, \bar u)$.

We note that after the renormalization $\psi^\ast_{\bar X,r}$ satisfies 
\begin{align*}
a^{ij}_r\pt_{ij} \psi^\ast_{\bar X,r}-r\pt_{p_ip_2}\tilde\MG_r\left(\frac{\nabla\psi^\ast_{\bar X,r}}{\bar y+ry}, \psi^\ast_{\bar X,r}\right)\frac{\pt_i\psi^\ast_{\bar X,r}}{\bar y+ry}=(\bar y+ry)f_r \quad\text{in }  \{\psi^\ast_{\bar X,r}>0\},
\end{align*}
where
\begin{align*}
	a^{ij}_r:=\pt_{p_ip_j}\tilde{\MG}_r\left(\frac{\nabla\psi^\ast_{\bar X,r}}{\bar y+ry}, \psi^\ast_{\bar X,r}\right)
\end{align*}
and
\begin{align*}
	f_r:=&-\pt_i\psi^\ast_{\bar X,r} \pt_{p_iz}\tilde{\MG}_r\left(\frac{\nabla\psi^\ast_{\bar X,r}}{\bar y+ry}, \psi^\ast_{\bar X,r}\right)
	+(\bar y+ry) \pt_z\tilde{\MG}_r\left(\frac{\nabla\psi^\ast_{\bar X,r}}{\bar y+ry}, \psi^\ast_{\bar X,r}\right).
\end{align*}
From \eqref{eq_rescale0} and \eqref{eq:BQ2} we conclude that there exist  constants $C_\gamma=C_\g(\g)>1$ and $C=C(\gamma,\epsilon,\bar u, \bar H)>0$ such that
\begin{align*}
	1\leq \frac{\tilde{\mathfrak b}^\ast}{\tilde{\mathfrak b}_\ast}\leq \frac{C_\gamma}{ \epsilon} \q\text{and}\q {\left|\frac{\pt_z\tilde{\MG}_r}{\tilde{\mathfrak b}_\ast}\right|+}
	\left|\frac{f_r}{\tilde{\mathfrak b}_\ast}\right|\leq Cr.
\end{align*}

The following comparison principle for elliptic equation \eqref{ellipticequ scale} is important for the proof of the Lipschitz regularity and nondegeneracy of minimizers. 

\begin{lemma}\label{comparison principle}
	Given a bounded domain $\tilde\D\subset \R^2$ and $\bar y\in(0,\bar H)$. For $r>0$, let $\psi^*\in H^1(\tilde \D)$ be a subsolution of the equation \eqref{ellipticequ scale} in the sense of
	\begin{equation*}\label{psi*_subsolution}
		\int_{\tilde\D}\left[\pt_{p_i}\tilde\MG_r\left(\frac{\n\psi^*}{\bar y+ry},\psi^*\right)\pt_i\eta+(\bar y+ry)\pt_z\tilde\MG_r\left(\frac{\n\psi^*}{\bar y+ry},\psi^*\right)\eta\right]\leq0
		\q\text{for all } \eta\in C_0^\infty(\tilde\D),
	\end{equation*} 
	and $\phi\in H^1(\tilde\D)$ be a solution of \eqref{ellipticequ scale} in the sense of
 \begin{equation*}\label{psi*_subsolution}
 	\int_{\tilde\D}\left[\pt_{p_i}\tilde\MG_r\left(\frac{\n\phi}{\bar y+ry},\phi\right)\pt_i\eta+(\bar y+ry)\pt_z\tilde\MG_r\left(\frac{\n\phi}{\bar y+ry},\phi\right)\eta\right]=0
 	\q\text{for all } \eta\in C_0^\infty(\tilde\D). 
 \end{equation*}
  Assume that $\psi^*\leq\phi$ on $\pt \tilde\D$. Then $\psi^*\leq\phi$ in $\tilde\D$, as long as $r\leq r^*$ for some $r^*>0$ sufficiently small depending on $\g,\, \v,\, \bar u,\, \bar H$, and $d:={\rm diam}\tilde\D$. 
\end{lemma}
\begin{proof}
	Take $\eta=(\psi^*-\phi)^+$. Then
	\begin{align*}
	&\int_{\tilde\D}\left[\pt_{p_i}\tilde\MG_r\left(\frac{\n\psi^*}{\bar y+ry},\psi^*\right)-\pt_{p_i}\tilde\MG_r\left(\frac{\n\phi}{\bar y+ry},\phi\right)\right]\pt_i\eta
	\\
	+&\int_{\tilde\D}(\bar y+ry)\left[\pt_z\tilde\MG_r\left(\frac{\n\psi^*}{\bar y+ry},\psi^*\right)-\pt_z\tilde\MG_r\left(\frac{\n\phi}{\bar y+ry},\phi\right)\right]\eta\leq 0.
	\end{align*}
	Using the convexity of $\tilde\MG_r$ in \eqref{eq:Gm_dpp} and the triangle inequality yields
	\begin{align*}
		&\int_{\tilde\D} \left[\pt_{p_i}\tilde\MG_r\left(\frac{\n\psi^*}{\bar y+ry},\psi^*\right)-\pt_{p_i}\tilde\MG_r\left(\frac{\n\phi}{\bar y+ry},\phi\right)\right]\pt_i\eta\\
		\geq& \tilde{\mathfrak b}_*\int_{\tilde\D}\frac{|\n\eta|^2}{\bar y+ry}+\int_{\tilde\D}\left[\pt_{p_i}\tilde\MG_r\left(\frac{\n\phi}{\bar y+ry},\psi^*\right)-\pt_{p_i}\tilde\MG_r\left(\frac{\n\phi}{\bar y+ry},\phi\right)\right]\pt_i\eta,
	\end{align*}
	where the last integral in the above inequality can be estimated from \eqref{eq_rescale0} as
	\begin{align*}
		&\int_{\tilde\D}\left[\pt_{p_i}\tilde\MG_r\left(\frac{\n\phi}{\bar y+ry},\psi^*\right)-\pt_{p_i}\tilde\MG_r\left(\frac{\n\phi}{\bar y+ry},\phi\right)\right]\pt_i\eta\\
		=&\int_{\tilde\D}\left[(\psi^*-\phi)\int_0^1\pt_{p_iz}\tilde\MG_r\left(\frac{\n\phi}{\bar y+ry},s\psi^*+(1-s)\phi\right)ds\right]\pt_i\eta
		\geq-\tilde\d r\int_{\tilde\D}|\n\eta|\eta.
	\end{align*}
	Similarly, using \eqref{eq_rescale0} and the triangle inequality yields
	$$\int_{\tilde\D}(\bar y+ry)\left[\pt_z\tilde\MG_r\left(\frac{\n\psi^*}{\bar y+ry},\psi^*\right)-\pt_z\tilde\MG_r\left(\frac{\n\phi}{\bar y+ry},\phi\right)\right]\eta
	\geq -\tilde\d r\int_{\tilde\D}(|\n\eta|\eta+r(\bar y+ry)\eta^2).$$
	Combining the above estimates together gives
	\begin{equation}\label{label_3}
		\int_{\tilde\D}\frac{|\n\eta|^2}{\bar y+ry}\leq \frac{2\tilde\d r}{\tilde{\mathfrak b}_*}\int_{\tilde\D}(|\n\eta|\eta+r(\bar y+ry)\eta^2).
	\end{equation}
	Since the domain $\tilde\D\subset\R^2$ is bounded, applying the Cauchy-Schwarz inequality and the Poincar\'{e} inequality to $\eta(x,\cdot)$ for each $x$ to \eqref{label_3}, one has
	$$\int_{\tilde\D}|\n\eta|^2\leq \frac{1}2\int_{\tilde\D}|\n\eta|^2+Cd^2(\bar H+rd)^2\frac{\tilde\d r^2}{\tilde{\mathfrak b}_*}\left(\frac{\tilde\d}{\tilde{\mathfrak b}_*}+1\right)\int_{\tilde \D}|\n\eta|^2,$$
	where $d:={\rm diam}\tilde\D$ and $C>0$ is a universal constant. Thus in view of \eqref{eq:BQ2}, if $r\leq r^*$ for some $r^*=r^*(\g,\v,\bar u,\bar H, d)$  sufficiently small, then $\n\eta=0$ in $\tilde\D$, which implies that $\eta=0$ in $\tilde\D$. This completes the proof of the lemma.
\end{proof}

Based on the above discussions, we can now use the same arguments as in  \cite{L2023_axi_small} to derive the following conclusions. The first two lemmas give the decay rate of a minimizer away from the free boundary.  

\begin{lemma}\label{linear growth}
	Let $\psi$ be a minimizer of \eqref{minimization problem}. 
	%There exists a constant $\bar r=\bar r(\gamma,\epsilon, \bar u,\bar H)\in(0,1/4)$ such that for any $\bar{X}=(\bar x,\bar y)\in\D\cap\{\psi<Q\}$ satisfying $${\rm dist}(\bar{X},\G_{\psi})\leq \bar r{\rm dist}(\bar{X},\pt\D),$$ there exists $C=C(\gamma,\epsilon, \bar u,\bar H?)>0$ such that $$Q-\psi(\bar{X})\leq C\Ld\bar y {\rm dist}(\bar{X},\G_{\psi}).$$ 
	Let $\bar{X}=(\bar x,\bar y)\in\D\cap\{\psi<Q\}$ satisfy
	${\rm dist}(\bar{X},\G_{\psi})\leq \bar r{\rm dist}(\bar{X},\pt\D)$, where $\bar r=\bar r(\gamma,\epsilon, \bar u, \frac{\Ld}{Q}, 
	\bar H)\in(0,1/4)$. Then there exists a constant  $C=C(\gamma,\epsilon, \bar u,
	\bar H)>0$ such that
	$$Q-\psi(\bar{X})\leq C\Ld\bar y {\rm dist}(\bar{X},\G_{\psi}).$$
\end{lemma}

\begin{lemma}\label{nondegeneracy}(Nondegeneracy) 
	Let $\psi$ be a minimizer of \eqref{minimization problem}. Then for any $\vartheta>1$ and any $0<a<1$, there exist positive constants $c_a,\, r_*>0$ depending on $\gamma,\, \epsilon,\, \bar u,\, \bar H,\, \vartheta$, and $a$,  such that for any $B_r(\bar{X})\subset\D$ with $\bar X=(\bar x,\bar y)$ and $r\leq \min\{r_*,c_a\frac{\Lambda}{Q}\}\bar y$, if
	\begin{equation*}%\label{nondegeneracy 1}
		\frac 1r\left(\dashint_{B_r(\bar{X})}|Q-\psi|^\vartheta\right)^{\frac 1\vartheta}\leq c_a\Ld \bar y,
	\end{equation*}
	then $\psi=Q$ in $B_{ar}(\bar{X})$.
\end{lemma}

The next proposition shows the Lipschitz regularity of minimizers.

\begin{proposition}\label{Lipschitz}
	Let $\psi$ be a minimizer of \eqref{minimization problem}, then $\psi\in C_{{\rm loc}}^{0,1}(\D)$. Moreover, for any connected domain ${\D'}\Subset \D$ containing a free boundary point, the Lipschitz constant of $\psi$ in $\D'$ is estimated by $C\Ld$, where $C$ depends on $\gamma,\, \epsilon,\, \bar u,\,\bar H, \, \frac{\Ld}{Q},\, \D'$ and $\D$.
\end{proposition}

\begin{remark}\label{Lipschitz boundary}
	By the boundary estimate for the elliptic equation, $\psi$ is Lipschitz up to the $C^{1,\a}$ portion $\Sigma\subset \pt \D\cap\{y>0\}$ as long as the boundary data $\psi^\sharp\in C^{0,1}(\D\cup \Sigma)$.  Moreover, if $\D'$ is a subset of $\overline{\D}\cap\{y>0\}$ with ${\D'}\cap \pt\D$ being $C^{1,\a}$, ${\D'}\cap \D$ is connected, and ${\D'}$ contains a free boundary point, then $|\n\psi|\leq C\Ld$ in ${\D'}$.
\end{remark}

\subsection{Regularity of the free boundary}

With the help of the Lipschitz regularity and nondegeneracy of minimizers, we obtain the following regularity of the free boundary. For the proof we refer to \cite[Section 4.4]{L2023_axi_small}.

\begin{proposition}%(\cite[Proposition 4.17]{L2023_axi_small})
Let $\psi$ be a minimizer of \eqref{minimization problem}. 
The free boundary $\G_\psi$ is locally $C^{k+1,\a}$ if $\MG(\bp,z)$ is $C^{k,\a}$ in its components ($k\geq1$, $0<\alpha<1$), and it is locally real analytic if $\MG(\bp,z)$ is real analytic.
\end{proposition}

In our case the function $\mathcal{G}(\bp,z)$ is $C^{1,1}$ in its components (cf. Section \ref{subsec_existence}), thus the free boundary is locally $C^{2,\alpha}$ for any $\alpha\in (0,1)$.

\section{Fine properties for the free boundary problem}\label{sec fine properties}

In this section, we first obtain the uniqueness and monotonicity of the solution to the truncated problem with specific boundary conditions. The monotonicity of the solution is crucial to prove the graph property of the free boundary, as well as the equivalence between the Euler system and the stream function formulation. Then we establish the continuous fit and smooth fit of the free boundary. That is, the free boundary fits the outlet of the nozzle in a continuous differentiable  fashion.

\subsection{Uniqueness and monotonicity of the minimizer to problem \eqref{variation problem}} 
We take a specific boundary value $\psi^\sharp_{\mu,R}$ on $\pt\Omega_{\mu,R}$. Let $b_\mu\in(1,\bar H)$ be such that $N(b_\mu)=-\mu$, where $N$ is defined in \eqref{nozzle}. Choose a point $(-\mu,b_\mu')$ with  $0<k_\mu:=b_\mu-b'_\mu<(b_\mu-1)/4$. 
Let $H_*:=H_*(\Lambda)$ be such that $\Ld H_*^2e^{1-H_*}=Q$ if $\Ld>Q$, and $H_*=1$ if $\Ld\leq Q$. Define
	\begin{equation*}
		\psi^\dag(y) :=\left\{
		\begin{aligned}
			&\min\left(\Ld y^2e^{1-y}, Q\right)\quad &\text{if } H_*<1,\\
			&Q y^2e^{1-y}\quad &\text{if } H_*=1.
		\end{aligned}
		\right.
	\end{equation*}
	Let $s\in(3/2,2)$ be a fixed constant.
Set
\begin{eqnarray}\label{psi0}
	\psi^\sharp_{\mu, R}(x,y):=
	\left\{ \begin{split}
		& 0 &&\text { if } x=-\mu,\, 0<y<b_\mu',\\
		& Q\left(\frac{y-b_\mu'}{k_\mu}\right)^{s} && \text{ if } x=-\mu,\, b_\mu'\leq y\leq b_\mu,\\
		& Q &&\text{ if } (x,y)\in \N\cup ([0,R]\times \{1\}),\\
		& \psi^\dag(y) &&\text{ if } x=R,\, 0<y<1,\\
		& 0 &&\text{ if } -\mu\leq x\leq R,\, y=0.
	\end{split}\right.
\end{eqnarray}
	Note that $\psi^\sharp_{\mu, R}$ is continuous and it satisfies $0\leq \psi^\sharp_{\mu, R}\leq Q$.

\begin{lemma}\label{psi0 sup-subsol}
	The function $\psi^\sharp_{\mu, R}(R,\cdot)$ is a supersolution to \eqref{EL equ} in $\Omega_{\mu, R}$. Moreover, if $k_\mu$ is sufficiently small depending on $\gamma$, $\bar u$, and $\bar H$, then $\psi^\sharp_{\mu, R}(-\mu, \cdot)$ is a subsolution to \eqref{EL equ} in $\Omega_{\mu, R}$.
\end{lemma}
\begin{proof}
	We first assume that $\Ld>Q$. Denote $\phi(y):=\Ld y^2e^{1-y}$. Straightforward computations give
		\begin{align*}
		\left(g_\v\left(\left|\frac{\phi'}{y}\right|^2,\phi\right)\frac{\phi'}{y}\right)'
		=&\left(g_\v+2\pt_tg_\v\left(\frac{\phi'}{y}\right)^2\right)\left(\frac{\phi''}{y}-\frac{\phi'}{y^2}\right)+\pt_zg_\v\frac{(\phi')^2}{y}\\
		=&\left(g_\v+2\pt_tg_\v\left(\frac{\phi'}{y}\right)^2\right)\Ld e^{1-y}(y-3)
		+\pt_zg_\v\frac{(\phi')^2}{y}.
	\end{align*}
	Note that $\pt_zg_\epsilon<0$ by the expression of $\pt_zg_\v$ in \eqref{gm_dz}, $\pt_tg>0$ (cf. \eqref{eq:g_dt_bound}),  and $\B'(z)\geq0$ (cf. \eqref{eq:u0_eps0_B}). 
	Using \eqref{beta1eps} and \eqref{supportG}, one obtains
	\begin{equation*}
		\left(g_\v\left(\left|\frac{\phi'}{y}\right|^2,\phi\right)\frac{\phi'}{y}\right)'-y\pt_zG_\v\left(\left|\frac{\phi'}{y}\right|^2,\phi\right)<0 \q \text{on }(0,H_*).
	\end{equation*}
	Since $Q$ is a supersolution to \eqref{EL equ}, the function $\min\{\Ld y^2e^{1-y},Q\}$ is a supersolution to \eqref{EL equ}. For the case $\Ld\leq Q$, the function $Qy^2e^{1-y}$ is also a supersolution to \eqref{EL equ}. Thus the function $\psi^\sharp_{\mu, R}(R,\cdot)$ is a supersolution to \eqref{EL equ} in $\Omega_{\mu, R}$. 
	
	Denote $\vp(y):=Q\left(\frac{y-b_\mu'}{k_\mu}\right)^{s}$.  Straightforward computations give
	\begin{align*}
		&\left(g_\v\left(\left|\frac{\vp'}{y}\right|^2,\vp\right)\frac{\vp'}{y}\right)'\\
		%=&\left(g_\v+2\pt_tg_\v\left(\frac{\vp'}{y}\right)^2\right)\left(\frac{\vp''}{y}-\frac{\vp'}{y^2}\right)+\pt_zg_\v\frac{(\vp')^2}{y}\\
		=&\left(g_\v+2\pt_tg_\v\left(\frac{\vp'}{y}\right)^2\right)\frac{Qs}{yk_\mu^2}\left(\frac{y-b_\mu'}{k_\mu}\right)^{s-2}\left(s-2+\frac{b_\mu'}{y}\right)
		+\pt_zg_\v\frac{(\vp')^2}{y}.
	\end{align*}
		It follows from Lemma \ref{lem:truncation_g}, the estimates  \eqref{eq:dzgm_dtgm}-\eqref{lable_2} and \eqref{g bound}, as well as \eqref{eq:upper_pzzG} that
\begin{align}\label{ineq10}
g_\v+2\pt_tg_\v\left(\frac{\vp'}{y}\right)^2\geq CB_*^{-\frac1{\gamma-1}},\quad |\pt_zg_\epsilon|\leq C\kappa_0 B_*^{\frac1{\gamma-1}}\pt_t g_\epsilon,\quad
	|\pt_z G_\epsilon|\leq C\kappa_0, 
\end{align}
where $C=C(\g)>0$. Recall that $Q\geq C(\g,\bar u)B_*^{\frac1{\g-1}}$ by \eqref{label_7}. Therefore, if $k_\mu=k_\mu(\gamma,\bar u,\bar H)$ is sufficiently small, the function $\vp$ satisfies
	\begin{equation*}
		\left(g_\v\left(\left|\frac{\vp'}{y}\right|^2,\vp\right)\frac{\vp'}{y}\right)'-y\pt_zG_\v\left(\left|\frac{\vp'}{y}\right|^2,\vp\right)>0\q \text{on }(b'_\mu,b_\mu).
	\end{equation*}
Since $0$ is a solution to \eqref{EL equ}, then $\psi^\sharp_{\mu,R}(-\mu,\c)$ is a subsolution to \eqref{EL equ} in $\Omega_{\mu, R}$. 
This finishes the proof of the lemma.
\end{proof}

By virtue of Lemma \ref{psi0 sup-subsol}, we can derive the following proposition (cf. \cite[Lemma 5.2 and Proposition 5.4]{L2023_axi_small}), which demonstrates that problem \eqref{variation problem} has a unique minimizer and the minimizer is monotone increasing in the $x$-direction.  In addition, this proposition gives a uniform (with respect to $\mu$ and $R$) estimate for minimizers  of the truncated problems. 

\begin{proposition}\label{psi monotonic}
Let $\psi_{\mu,R,\Ld}$ be a minimizer of problem \eqref{variation problem} in $\O_{\mu,R}$ with the boundary value $\psi^\sharp_{\mu,R}$ constructed in \eqref{psi0}. If $k_\mu=k_\mu(\gamma, \bar u, \bar H)$ is sufficiently small, then
\begin{equation}\label{psi bound}
\psi^\sharp_{\mu,R}(-\mu,y)<\psi_{\mu,R,\Ld}(x,y)\leq\psi^\sharp_{\mu,R}(R,y) \q\text{for all } (x,y)\in \O_{\mu,R}.
\end{equation}
Moreover, $\psi_{\mu,R,\Ld}$ is the unique minimizer of problem \eqref{variation problem}, and $\pt_{x}\psi_{\mu,R,\Ld}\geq0$ in $\O_{\mu,R}$.
\end{proposition}

\subsection{Fine properties of the free boundary}\label{subsec continuousfit}
Denote the set of the free boundary points which lie strictly below $\{y=1\}$ as 
\begin{equation}\label{Gamma truncated}
	\G_{\mu,R,\Ld}:=\pt\{\psi<Q\}\cap\{(x,y): (x,y)\in (-\mu,R]\times (0,1)\}.
\end{equation}
The following graph property of the free boundary $\G_{\mu,R,\Lambda}$ is obtained from Proposition \ref{psi monotonic} and the uniqueness of the Cauchy problem for the elliptic equation (cf. (\cite[Lemma 5.6]{L2023_axi_small})). 

\begin{proposition}\label{y graph}(\cite[Proposition 5.7]{L2023_axi_small})
Let $\G_{\mu,R,\Ld}$ be defined in \eqref{Gamma truncated}.
%\begin{equation}\label{Gamma truncated}\G_{\mu,R,\Ld}=\pt\{\psi<Q\}\cap\{(x,y): (x,y)\in (-\mu,R]\times (0,1)\}.\end{equation}
If $\G_{\mu,R,\Ld}\neq\emptyset$, then it can be represented as a graph of a function of $y$, i.e., there exists a function $\ur_{\mu,R,\Ld}$ and $\ul H_{\mu,R,\Ld}\in[H_*,1)$ such that
$$\G_{\mu,R,\Lambda}=\{(x,y): x=\ur_{\mu,R,\Ld}(y),~ y\in(\ul H_{\mu,R,\Ld},1)\}.$$
Furthermore, $\ur_{\mu,R,\Ld}$ is continuous, $-\mu<\ur_{\mu,R,\Ld}\leq R$ and $\lim_{y\to1-}\ur_{\mu,R,\Ld}(y)$ exists.
\end{proposition}

%\subsection{Continuous fit, smooth fit}\label{subsec fit}
%In this section, we establish the continuous fit and smooth fit between the free boundary and the nozzle. Then we remove the truncations of the domain by letting $\mu, R\rightarrow \infty$ to obtain a limiting solution with fine properties in the whole domain. Finally, we study the asymptotic behavior of the solution as $x\rightarrow \pm\infty$. The far field asymptotic behavior plays an important role in proving the equivalence between Problem \ref{probelm 1} and Problem \ref{pb2}.

In order to show that there exists a suitable $\Ld>0$ such that the free boundary $\G_{\mu,R,\Ld}$ joins the outlet of the nozzle $\N$ as a continuous curve, i.e., the free boundary satisfies the continuous fit property, we give the next lemma. The proof of the lemma is based on the nondegeneracy of the minimizer and the comparison principle (cf. the proof for \cite[Lemma 6.2]{L2023_axi_small}). 

\begin{lemma}\label{continuous fit}
There exists a constant $C=C(\gamma,\epsilon,\bar u,\bar H)>0$ such that the following statements hold.

\rm{(i)} If $\Lambda>{CQ}$, then the free boundary $\G_{\mu,R,\Ld}$ is nonempty and $\Upsilon_{\mu,R,\Ld}(1)<0$.

\rm{(ii)} If $\Lambda<{C^{-1}Q}$, then $\Upsilon_{\mu,R,\Ld}(1)>0$.
\end{lemma}
%\begin{proof}
%(i) The proof is the same as that in \cite[Lemma 6.2]{L2023_axi_small}.

%Suppose that $\G_{\mu,R,\Ld}$ is nonempty and $\ur_{\mu,R,\Ld}(1)\geq0$.  We claim that $\Ld/Q$ cannot be too large. Let $r_*=r_*(\gamma,\epsilon,\bar u,\bar H)$ and $c_{1/2}=c_{1/2}(\gamma,\epsilon,\bar u,\bar H)$ be the constants determined in Lemma \ref{nondegeneracy} with $p=2$ and $a=1/2$. Since $\G_{\mu,R,\Ld}$ connects $(\ur_{\mu,R,\Ld}(1),1)$ to $(R,\ul H_{\mu,R,\Ld})$, there exists a free boundary point $X_*=(x_*, y_*)$ (with $y_*>1/4$) and $r\in (0,r_*/4)$ independent of $\Lambda$ or $Q$ such that $B_{ry_*}(X_*)\subset \O_{\mu,R}\cap\{y>1/4\}$ or $B_{r y_*}(X_*)\subset \{x>0,y>1/4\}$.  If $c_{1/2}\Lambda/Q\leq r_*$, then we are done. If $c_{1/2}\Lambda/Q>r_*$, then it follows from the nondegeneracy lemma (Lemma \ref{nondegeneracy}) that
%\begin{align*}
%	\frac{Q}{ry_*}\geq \frac{1}{ry_*}\left(\dashint_{B_{ry_*}(X_*)}|Q-\psi|^2\right)^{1/2}\geq c_{1/2}\Ld y_*.
%\end{align*}
%Since $r$ is fixed and $y_*>\frac14$, there is a contradiction if ${\Lambda}/Q$ is large.

%If the free boundary is empty, then one can get a contradiction by taking any $X_*=(x_*,1)$ with $x_*>0$ and using the linear nondegeneracy for nonnegative solutions along flat boundaries. This finishes the proof of (i).

%(ii) The proof for (ii) is the same as that in \cite[Lemma 6.2]{L2023_axi_small}.
%\end{proof}

Since $\psi_{\mu,R,\Ld}$ and $\ur_{\mu,R,\Ld}$ depend on $\Ld$ continuously (cf. \cite[Lemma 6.1]{LSTX2023} and \cite[Lemma 6.1]{L2023_axi_small}), the continuous fit of the free boundary follows from Lemma \ref{continuous fit}.

\begin{corollary}\label{Ld determine}
There exists $\Ld=\Ld(\mu,R)>0$ such that $\ur_{\mu,R,\Ld}(1)=0$. Furthermore,
$$C^{-1}Q\leq \Ld(\mu,R)\leq CQ$$
for positive constant $C$ depending on $\g$, $\v$, $\bar u$, and $\bar H$,  but independent of $\mu$ and $R$.
\end{corollary}

The continuous fit implies the following smooth fit of the free boundary. 

\begin{proposition}\label{smooth fit}(\cite[Proposition 6.4]{L2023_axi_small})
Let $\psi$ be a solution obtained in Corollary \ref{Ld determine}. Then $\N\cup\G_{\mu,R,\Ld}$ is $C^1$ in a neighborhood of the point $A=(0,1)$, and $\n\psi$ is continuous in a $\{\psi<Q\}$-neighborhood of $A$.
\end{proposition}

%%%%%%%%%%%%%%%%%%%%%%%%%%%%%%%%%%%%%%%%%%%%%%%%%%%%%%%%%%%%%%%%%%%%%%%%%%%%%%%%%%%%%%%%%%%%%%%%%%%%%%%%%%%%%%%%%%%%%%%%%%%%%%%%%%%%%%%%

\section{Existence of solutions to the jet problem}\label{sec remove truncation}

In this section, we first remove the domain and subsonic truncations, and then study the far fields behavior of the jet flows. Consequently, we obtain the existence of subsonic jet flows which satisfy all properties in Problem \ref{pb2}.

\subsection{Remove the domain truncations}
In this subsection, we remove the truncations of the domain in Section \ref{sec variational formulation} by letting $\mu, R\to\infty$ to get a limiting solution in $\Omega$, which is bounded by $\N_0$ and $\N\cup([0,\infty)\times\{1\})$. The limiting solution inherits the properties of solutions in the truncated domain due to the uniform estimates (with respect to $\mu$ and $R$). The proof is similar to \cite[Proposition 7.1]{LSTX2023}, so we shall not repeat it here. 
%Moreover, it satisfies certain asymptotic behavior at the upstream and downstream.

\begin{proposition}\label{psi determine}
Let the nozzle boundary $\N$ defined in \eqref{nozzle} satisfy \eqref{nozzle condition}. Given an incoming axial velocity  $\bar u\in C^{1,1}([0,\bar H])$ satisfying \eqref{ubar_condition} and a mass flux $Q>\tilde Q$ where $\tilde Q$ is defined in \eqref{def:Q_*}.
Let $\psi_{\mu,R,\Ld}$ be a minimizer of the problem \eqref{variation problem}
%$$\inf_{\phi\in \K_{\psi^\sharp_{\mu,R}}}J^\v_{\mu,R,\Ld}(\phi)$$
with $\psi^\sharp_{\mu,R}$ defined in \eqref{psi0}. Then for any $\mu_j, R_j\to\infty$, there is a subsequence (still labeled by $\mu_j$ and $R_j$) such that $\Ld_j:=\Ld(\mu_j,R_j)\to \Ld_\infty$ for some $\Ld_\infty\in (0,\infty)$ and $\psi_{\mu_j,R_j,\Ld_j}\to \psi_\infty$ in $C^{0,\a}_{\rm loc}(\O)$ for any $\a\in (0,1)$. Furthermore, the following properties hold. 
\begin{itemize}
	\item [(i)] The function $\psi:=\psi_\infty$ is a local minimizer for the energy functional, i.e., for any $\D\Subset \Omega$,  one has  $J^\epsilon(\psi)\leq J^\epsilon(\varphi)$ for all  $\varphi=\psi$ on {$\partial D$}, where
	\begin{align*}
		J^\epsilon(\varphi):=\int_{\D}y\bigg[G_\v\bigg(\left|\frac{\n\phi}{y}\right|^2,\phi\bigg)+\ld^2_{\v,\infty}\chi_{\{\phi<Q\}}\bigg]dX,
		\quad \lambda_{\epsilon,\infty}:=\sqrt{\Phi_\epsilon(\Lambda_\infty^{{2}},Q)}.
	\end{align*}
    In particular, $\psi$ solves
		\begin{equation}\label{eq_limiting_sol}
			\left\{
			\begin{aligned}
				&\n\c\bigg(g_\v\bigg(\left|\frac{\n\psi}{y}\right|^2,\psi\bigg)\frac{\n\psi}{y}\bigg)-y\pt_zG_\v\bigg(\left|\frac{\n\psi}{y}\right|^2,\psi\bigg)=0 &&\text{ in } \mathcal{O},\\
				&\psi =0 &&\text{ on } \N_0,\\
				&\psi =Q &&\text{ on } \N\cup \Gamma_{\psi},\\
				&\left|\frac{\nabla \psi}y\right| =\Lambda &&\text{ on } \Gamma_{\psi},
			\end{aligned}
			\right.
		\end{equation}
		where $\mathcal{O}:=\Omega\cap \{\psi<Q\}$ is the flow region, $\Gamma_{\psi}:=\pt\{\psi<Q\}\backslash \N$ is the free boundary, and $\pt_zG_\v(|\n\psi/y|^2,\psi)$ satisfies \eqref{Gdz expression} in the subsonic region $|\nabla\psi/y|^2\leq (1-\epsilon)\t_c(\mathcal{B}(\psi))$.
	\item [(ii)] The function $\psi$ is in $C^{2,\alpha}(\mathcal{O})\cap C^1(\overline{\mathcal{O}})$ and it satisfies $\pt_{x}\psi\geq 0$ in $\Omega$.
	\item[(iii)] The free boundary $\Gamma_{\psi}$ is given by the graph $x=\Upsilon(y)$, where $\Upsilon$ is a $C^{2,\alpha}$ function as long as it is finite.
	\item[(iv)] At the orifice $A=(0,1)$ one has $\lim_{y\rightarrow 1-}\Upsilon(y)=0$. Furthermore, $\N\cup\Gamma_{\psi}$ is $C^1$ around $A$, in particular, $N'(1)=\lim_{y\rightarrow 1-}\Upsilon'(y)$.
	\item[(v)] There is a constant $\ubar H\in (0,1)$ such that $\Upsilon(y)$ is finite if and only if $y\in (\ubar H, 1]$, and $\lim_{y\rightarrow \ubar H+} \Upsilon(y)=\infty$. Furthermore, there exists an $\bar R>0$ sufficiently large, such that $\Gamma_{\psi}\cap \{x>\bar R\}= \{(x, f(x)): \bar R<x<\infty\}$ for some $C^{2,\alpha}$ function $f$ and $\lim_{x\rightarrow \infty}f'(x)=0$.
\end{itemize}
\end{proposition}

The following proposition shows that the limiting solution $\psi:=\psi_\infty$ satisfies $\pt_x\psi>0$ inside the flow region. This property implies the negativity of the radial velocity $v$ (cf. \eqref{psi gradient}), which ensures that the streamlines have a simple topological structure.

\begin{proposition}\label{psix strictly pro}(\cite[Proposition 7.2]{L2023_axi_small}) 
The solution $\psi:=\psi_\infty$ obtained in Proposition \ref{psi determine} satisfies $\pt_{x}\psi>0$ in $\MO:=\O\cap\{\psi<Q\}$.
\end{proposition}

\subsection{Remove the subsonic truncation}
In this subsection, we remove the subsonic truncation introduced in Section \ref{subsec_subsonic}, provided the mass flux $Q$ is sufficiently large. 
After removing the subsonic truncation, the limiting solution obtained in Proposition \ref{psi determine} is actually a solution of \eqref{eq:pb2}. 
 
\begin{proposition}\label{subsonic pro}
	Suppose $Q>{Q}^*$ for some ${Q}^*\geq \tilde Q$ sufficiently large depending on $\bar u$, $\gamma$, $\epsilon$, and the nozzle, where $\tilde Q$ is defined in \eqref{def:Q_*}. Let $\psi:=\psi_\infty$ be a limiting solution obtained in Proposition \ref{psi determine}. Then
	\begin{equation}\label{subsonic}
		\left|\frac{\n\psi}{y}\right|^2\leq (1-\epsilon)\t_c(\B(\psi)),
	\end{equation}
	where $\t_c$ is defined in \eqref{tc}. 
\end{proposition}
\begin{proof}
	In the flow region $\mathcal O:=\Omega\cap\{\psi<Q\}$, $\psi$ satisfies the following equation of nondivergence form
	\begin{equation*}\label{eqnondiv}
		\mathfrak a_\epsilon^{ij}\bigg(\frac{\n\psi}{y},\psi\bigg) \pt_{ij}\psi=\mathfrak b\bigg(\frac{\n\psi}{y},\psi\bigg)+y^2\tilde{G}_\epsilon\bigg(\left|\frac{\n\psi}{y}\right|^2,\psi\bigg),
	\end{equation*}
	where the matrix
	$$(\mathfrak a_\epsilon^{ij}):=g_\epsilon\bigg(\left|\frac{\n\psi}{y}\right|^2,\psi\bigg)I_2+2\pt_t g_\epsilon\bigg(\left|\frac{\n\psi}{y}\right|^2,\psi\bigg)\frac{\n\psi}{y}\otimes\frac{\n\psi}{y}$$
	is symmetric with the eigenvalues
	\begin{align*}
		\beta_{0,\epsilon}  =g_\epsilon\bigg(\left|\frac{\n\psi}{y}\right|^2,\psi\bigg)\quad \text{and}\quad
		\beta_{1,\epsilon} =g_\epsilon\bigg(\left|\frac{\n\psi}{y}\right|^2,\psi\bigg)+2\pt_t g_\epsilon\bigg(\left|\frac{\n\psi}{y}\right|^2,\psi\bigg)\left|\frac{\n\psi}{y}\right|^2,
	\end{align*}
	and 
	\begin{align*}
		&\mathfrak b:=\beta_{1,\v}\frac{\pt_y\psi}{y},\q
		%\bigg[g_\v\bigg(\left|\frac{\n\psi}{y}\right|^2,\psi\bigg)+2\pt_tg_\v\bigg(\left|\frac{\n\psi}{y}\right|^2,\psi\bigg)\left|\frac{\n\psi}{y}\right|^2\bigg],\\
		\tilde{G}_\epsilon:%\bigg(\left|\frac{\n\psi}{y}\right|^2,\psi\bigg)
		=-\pt_z g_\epsilon\bigg(\left|\frac{\n\psi}{y}\right|^2,\psi\bigg) \left|\frac{\n\psi}{y}\right|^2+\pt_zG_\epsilon\bigg(\left|\frac{\n\psi}{y}\right|^2,\psi\bigg).	
	\end{align*}

	It follows from Lemma \ref{lem:truncation_g}, the estimate of $\pt_zG_\v$ in \eqref{eq:upper_pzzG}, and \eqref{label_7} that there exist constants $C_\g=C_\g(\g)>0$ and $C=C(\g,\bar u)>0$ such that
	\begin{equation}\label{subsonic_eq}
		1\leq \frac{\beta_{1,\epsilon}}{\beta_{0,\epsilon}}\leq C_\gamma \epsilon^{-1}
		\quad\text{and}\quad
		%\tilde C_\gamma^{-1}\bar\rho\leq
		\frac{|\tilde G_\epsilon|}{\beta_{0,\epsilon}}\leq CQ\epsilon^{-1}.
	\end{equation}
	With \eqref{subsonic_eq} at hand, one can use similar arguments as in \cite[Proposition 7.3]{L2023_axi_small} to get
	\begin{equation}\label{eq:subsonic}
	\left\|\frac{\n\psi}{y}\right\|_{L^{\infty}(\overline{\mathcal O})}\leq 
		C(1+Q),
	\end{equation}
	where $C$ depends on $\gamma,\epsilon,\bar u$, and the nozzle.
	Note that the right-hand side of \eqref{eq:subsonic} is $O(Q)$ for $Q\gg 1$,  and $\t_c(\mathcal B(\psi))=O(Q^{\gamma+1})$ for $Q\gg 1$. Since $\gamma>1$, we obtain the desired conclusion.
\end{proof}
%\begin{proof}
%Given $X_0=(x_0,y_0)\in\MO:=\O\cap\{\psi<Q\}$. If $y_0\geq 1/4$, by elliptic estimates $$\left|\frac{\n\psi}{y}(X_0)\right|\leq C(\|\psi\|_{L^\infty(\MO)}+\|\pt_zG_\v\|_{L^\infty(\MO)}),$$ where $C$ depends on $B_*$, $S_*$, $\g$ and the nozzle. If $y_0\leq 1/4$, define $\psi_0(X)=\psi(X_0+rX)/r^2$ with $r=y_0/2$. Then $\psi_0$ satisfies $$\n\c\left(g_\v\left(\left|\frac{\n\psi_0}{y+2}\right|^2,\frac{y_0^2}{4}\psi_0\right)\frac{\n\psi_0}{y+2}\right)-\frac{y_0}2(y+2)\pt_z G_\v\left(\left|\frac{\n\psi_0}{y+2}\right|^2,\frac{y_0^2}{4}\psi_0\right)=0,$$ so that
%$$\left|\frac{\n\psi}{y}(X_0)\right|=\frac12|\n\psi_0(0)|\leq C(\|\psi_0\|_{L^\infty(B_1)}+\|\pt_zG_\v\|_{L^\infty(\MO)})$$
%for some $C$ depending on $B_*$, $S_*$, $\g$ and the nozzle. Note that $0\leq\psi\leq Q$ in $\MO$ and $0<\psi_0\leq CQ$, which follows from the definition of $\psi_0$, $0<\psi\leq\Ld y^2e^{1-y}$ (cf. Lemma \ref{psi bound lem}) as well as Corollary \ref{Ld determine}. Moreover, by \eqref{G dz} it holds
%$$\|\pt_zG_\v\|_{L^\infty(\MO)}\leq \d Q.$$
%Hence by Proposition \ref{incoming data} and \eqref{k0}, one has   \eqref{subsonic} provided %$\k$ in \eqref{BS condition2} sufficiently small and $Q\in(Q_*,\hat Q)$ with $\hat Q$ sufficiently small depending on $B_*$, $S_*$, $\g$ and the nozzle. This completes the proof of the proposition.
%\end{proof}

\subsection{Far fields behavior}\label{sec asymptotic}
This subsection devotes to the asymptotic behavior of the jet flows as $x\rightarrow \pm\infty$. It plays an important role in proving the equivalence between the Euler system and the stream function formulation.

\begin{proposition}\label{pro_asymptotic_upstream} 
Given a mass flux $Q>{Q}^*$, where $Q^*$ is defined in Proposition \ref{subsonic pro}. 
	Let $\psi:=\psi_\infty$ be the solution obtained in Proposition \ref{psi determine} with $\Ld:=\Ld_\infty$. Then for any $\alpha\in(0,1)$, as $x\to-\infty$,
	\begin{equation}\label{psi upstream}
		\psi(x,y)\to\bar{\psi}(y):=\bar{\r}\int_0^{y}s\bar{u}(s)ds
		\q \text{in } C_{{\rm loc}}^{2,\alpha}(\R\times(0,\bar H)),
	\end{equation}
where $\bar\rho$ is the upstream density and $\bar u$ is the upstream axial velocity; as $x\to+\infty$,
	\begin{equation}\label{psi downstream}
		\psi(x,y)\to\ul\psi(y):=\ul\r\int_0^ys\ul u(s)ds,	\q \text{in } C_{{\rm loc}}^{2,\alpha}(\R\times(0,\ubar H)),
	\end{equation}
where $\ubar\rho>0$ is the downstream density, $\ubar u\in C^{1,\alpha}((0,\ubar H])$ is the downstream (positive) axial velocity, and $\ubar H>0$ is the downstream asymptotic height. Moreover, $\ubar\rho$, $\ubar u$ and $\ubar H$ are uniquely determined by $\bar u$, $Q$, $\bar H$, $\g$, and $\Ld$.
\end{proposition}
\begin{proof}
%By Proposition \ref{subsonic pro}, the function $\psi$ satisfies \eqref{subsonic}.

(i) \emph{Upstream asymptotic behavior.} Let $\psi^{(-n)}(x,y): =\psi(x-n,y)$, $n\in\mathbb Z$. Since the nozzle is asymptotically horizontal with the height $\bar H$, there exists a subsequence $\psi^{(-n)}$ (relabeled) converges to a function $\hat\psi$ in $C_{\rm loc}^{2,\alpha}(\R\times[0,\bar H))$ for any $\alpha\in(0,1)$, where $\hat\psi$ satisfies \eqref{subsonic} and solves the Dirichlet problem in the infinite strip
\begin{equation}\label{system upstream}\begin{cases}
		\begin{split}
			&\n\c\bigg(g_\v\bigg(\bigg|\frac{\n\hat\psi}{y}\bigg|^2,\hat\psi\bigg)\frac{\n\hat\psi}{y}\bigg)-y\pt_z G_\v\bigg(\bigg|\frac{\n\hat\psi}{y}\bigg|^2,\hat\psi\bigg)=0 \q \text{ in } \R\times(0,\bar H),\\
			&\hat\psi=Q \ \text{ on } \R\times\{\bar H\},\q \hat\psi=0 \ \text{ on } \R\times\{0\}
			%\q 0\leq\hat\psi\leq Q \ \text{ in } \R\times[0,\bar H].
\end{split}\end{cases}\end{equation}
and it satisfies $0\leq\hat\psi\leq Q$ in $\R\times[0,\bar H]$.
On the other hand, it follows from \eqref{det_0} and the energy estimates (cf. \cite[Proposition 3]{DD2016_axially_large_vorticity}) that the problem \eqref{system upstream} has a unique solution.
Since $\hat\psi$ satisfies \eqref{subsonic} and the function $\bar\psi$ defined in \eqref{psi upstream} satisfies the equation in \eqref{eq:pb2}, in view of \eqref{Gdz expression} and \eqref{gm_g} one has $\hat\psi(x,y)=\bar \psi(y)$ in $\R\times[0,\bar H]$.
This proves the asymptotic behavior of the flows at the upstream.

(ii) \emph{Downstream asymptotic behavior.} Let $\psi^{(n)}(x,y):=\psi(x+n,y)$, $n\in\mathbb Z$. By the $C^{2,\a}$ regularity of the free boundary and the boundary regularity for elliptic equations, there exists a subsequence $\psi^{(n)}$ (relabeled) converging to a function $\ubar\psi$ which satisfies \eqref{subsonic} and solves
\begin{equation}\label{system downstream}\begin{cases}\begin{split}
			&\n\c\bigg(g_\v\bigg(\bigg|\frac{\n\ubar\psi}{y}\bigg|^2,\ubar\psi\bigg)\frac{\n\ubar\psi}{y}\bigg)-y\pt_z G_\v\bigg(\bigg|\frac{\n\ubar\psi}{y}\bigg|^2,\ubar\psi\bigg)=0 &&\text { in } \R\times(0,\ul H),\\
			&\ubar\psi=Q,\q \frac1y\pt_{y}\ubar\psi=\Ld &&\text{ on } \R\times\{\ul H\},\\
			&\ubar\psi=0  && \text{ on }\R\times\{0\}
			%&0\leq\tilde\psi\leq Q &&\text { in } \R\times[0,\ul H].
\end{split}\end{cases}\end{equation}
with $0\leq\ubar\psi\leq Q$ in $\R\times[0,\ul H]$.
%\textcolor{red}{In view of Proposition \ref{downstreamunique pro}, the problem \eqref{system downstream} has a solution.}
It follows from \cite[Lemma 5.6]{L2023_axi_small} that the solution to the above Cauchy problem is unique for given positive constants $Q$, $\Ld$, and $\ul H$.
Besides, applying the energy estimates (cf. \cite[Proposition 3]{DD2016_axially_large_vorticity}) one can conclude that $\pt_{x}\ubar\psi=0$, that is, $\ubar\psi$ is a one-dimensional function.
Thus in view of \eqref{psi gradient}, the radial velocity $v\to0$ as $x\to\infty$. The downstream density $\ubar\rho$ (which is a constant by using similar arguments as in \cite[Remark 1.1]{CDXX2016_past_wall} and \cite[Remark 1.1]{L2023_axi_small}) and the downstream axial velocity $\ubar u$ are determined by
\begin{equation}\label{ubar_rho}
	\ubar\rho=\frac1{g(\Ld^2,Q)} \q\text{and}\q \ubar u(y)=\frac{\ubar\psi'(y)}{y\ubar\rho}
\end{equation}
respectively, where $g$ is the function defined in \eqref{eq:branch}. 
Consequently, $\ubar\psi$ can be expressed as in \eqref{psi downstream}. 

Note that $\ubar \rho \in (\varrho_c(\B(z)), \varrho_m(\B(z))]$ for all $z\in[0,Q]$, where $\varrho_c$ and $\varrho_m$ are defined in \eqref{defrhoc}. Hence $\ubar\rho\leq \varrho_m(B_*)$, where $B_*=\B(0)$ is the lower bound of the Bernoulli function $\B$, cf. \eqref{eq:B} and \eqref{eq:u0_eps0_B}. According to the Bernoulli law, one has 
\begin{equation}\label{ubar_lowerbd}
\ubar u(y)=\sqrt{2(\B(\ubar\psi(y))-h(\ubar{\rho}))}\geq\sqrt{2(\B(0)-h(\ubar{\rho}))}=\ubar u(0)\geq0.
\end{equation}
We claim that $\ubar u(0)>0$. Actually, suppose $\ubar u(0)=0$, then by the Bernoulli law one has 
$$\frac{\bar u^2(0)}2+h(\bar\rho)=\B(0)=h(\ubar\rho),$$
thus $\ubar\rho>\bar\rho$. If we can prove $\ubar\rho<\bar\rho$, then there is a contradiction so that the claim holds true. For this we let $\th(y)$ be the position at downstream if one follows along the streamline with asymptotic height $y$ at the upstream, i.e., $\th:[0,\bar H]\to[0,\ul H]$ satisfies
\begin{equation}\label{theta def}
	\ul\psi(\th(y))=\bar\psi(y),\q y\in[0,\bar H],
\end{equation}
where $\bar\psi$ and $\ul\psi$ are defined in \eqref{psi upstream} and \eqref{psi downstream}, respectively.
Then the map $\th$ satisfies
\begin{align}\label{theta}
	\begin{cases}
		\th'(y)=\frac{y\bar{\r}\bar{u}(y)}{\th(y)\ul{\r}\ul{u}(\th(y))},\\
		\th(0)=0,
\end{cases}\end{align}
where $\ubar\rho>0$ is given by \eqref{ubar_rho} and 
\begin{equation}\label{ubar_u_theta}
	\ubar u(\theta(y))=\sqrt{2(\B(\bar\psi(y)) - h(\ubar{\rho}))}.
\end{equation}
Calculating directly gives that the map
$$\ubar\rho\mapsto \Theta(\ubar\rho):=\int_{0}^{\bar H}\frac{y\bar{\r}\bar{u}(y)}{\th(y)\ul{\r}\sqrt{2(\B(\bar\psi(y)) - h(\ubar{\rho}))}}dy$$
is monotone increasing for $\ubar \rho \in (\varrho_c(\B(z)), \varrho_m(\B(z))]$. Since $\Theta(\bar\rho)>\bar H>\ubar H$ (note that $\theta(y)<y$ for $y\in(0,\bar H]$ since $v<0$, which is obtained from Proposition \ref{psix strictly pro} and \eqref{psi gradient}), one gets $\ubar\rho<\bar\rho$. Thus $\ubar u(0)>0$. In view of \eqref{ubar_lowerbd},  $\ubar u$ has a positive lower bound. 

Now substituting \eqref{ubar_u_theta} into \eqref{theta} yields that $\th(y)$, in particular $\ul H=\theta(\bar H)$, is uniquely determined. This finishes the proof of the proposition.
\end{proof}

In view of Propositions \ref{psi determine}--\ref{pro_asymptotic_upstream}, we obtain the existence of solutions to {Problem} \ref{pb2} when the mass flux $Q$ is sufficiently large. Then the existence of solutions to Problem \ref{probelm 1} follows from Remark \ref{rmk_pb2}.

\section{Uniqueness of the outer pressure}\label{sec uniqueness}
In this section, we show that for given axial velocity $\bar u$ satisfying \eqref{ubar_condition} and mass flux $Q>Q^*$ at the upstream, where $Q^*$ is defined in Proposition \ref{subsonic pro}, there is a unique momentum $\Ld$ on the free boundary such that Problem \ref{pb2} has a solution. In view of \eqref{Ld} and the proof of Lemma \ref{g}(i), the uniqueness of $\Ld$ implies the uniqueness of the downstream density $\ubar \rho$ and the outer pressure $\ul p$. More precisely, the constant $\ubar \r$ in \eqref{Ld} is determined by $\ubar\r=1/g(\Ld^2,Q)$, where $g$ is the function defined in Lemma \ref{g}(i), hence $\ul p=\ubar\r^\g/\g$ is uniquely determined.
\begin{proposition}\label{uniqueness pro}
Suppose that $(\psi_i,\G_i,\Ld_i)$ $(i=1,2)$ are two uniformly subsonic solutions to Problem \ref{pb2}, then $\Ld_1=\Ld_2$.
\end{proposition}
\begin{proof}
Let $(\ubar\rho_i, \ubar u_i, \ubar H_i)$, $i=1,2$, be the associated downstream asymptotic states. 
Without loss of generality, assume that $\Lambda_1 <\Lambda_2$. By Bernoulli's law, cf. \eqref{eq:psi}, along the free boundaries $\Gamma_1$ and $\Gamma_2$, one has 
	$$\frac{\Ld_1^2}{2\ul\r_1^2}+h(\ubar\rho_1)=\B(Q)
	=\frac{\Ld_2^2}{2\ul\r_2^2}+h(\ubar\rho_2).$$
	Since for subsonic solutions,  $\ubar\rho_i=\varrho(\Ld_i^2,Q)   (i=1,2)$ is monotone decreasing with respect to $\Ld_i$, where $\varrho$ is defined in \eqref{varrho_def}, one gets $\ul\r_1>\ul\r_2$. Let $\theta_i(y)$ ($i=1$, $2$) be defined in \eqref{theta} associated with $(\ubar\rho_i, \ubar u_i)$ ($i=1$, $2$). 
	It follows from \eqref{psi gradient} and \eqref{eq:psi} that
	$$\frac{(\ul\r_1\ul u_1)^2(\th_1(y))}{2\ul\r_1^2(\th_1(y))}+h(\ul\r_1(\th_1(y)))=\B(\bar\psi(y))
	=\frac{(\ul\r_2\ul u_2)^2(\th_2(y))}{2\ul\r_2^2(\th_2(y))}+h(\ul\r_2(\th_1(y))).$$
	Thus $\ubar\rho_i=\varrho((\ubar\rho_i\ubar u_i)^2(\theta_i(y)), \bar\psi(y)) (i=1,2)$.
	Using $\ubar \rho_1>\ubar\rho_2$ and the monotone decreasing property of $\varrho(t,z)$ with respect to $t$ again yields
	\begin{equation*}
		\ubar\rho_1\ubar u_1(\theta_1(y))<\ubar \rho_2\ubar u_2(\theta_2(y)).
	\end{equation*}
	In view of \eqref{theta} this implies
	\begin{equation}\label{theta ineq1}
	(\th_1^2(y))'>(\th_2^2(y))'.
\end{equation}
	Note that $\theta_1(0)=\theta_2(0)=0$. Integrating \eqref{theta ineq1} on {$[0,\bar H]$} one has 
	$\ubar H_1=\theta_1(\bar H)>\theta_2(\bar H)=\ubar H_2$.
	
Let $\MO_i$ be the domain bounded by $\N_0$, $\N$ and $\G_i$. Since $\ul H_1>\ul H_2$ and that $\N\cup \G_i$ is a $y$-graph, one can find a $\tau\geq0$ such that the domain $\MO_1^{\tau}=\{(x,y):(x-\tau,y)\in\MO_1\}$ contains $\MO_2$. Let $\tau_*$ be the smallest number such that $\MO^{\tau_*}_1$ contains $\MO_2$ and they touch at some point $(x_*,y_*)\in\Gamma_2$. Define  $\psi_1^{\tau_*}(x,y):=\psi_1(x-\tau_*,y)$. 
Now we prove $\psi_1^{\tau_*}\leq \psi_2$  in $\mathcal O_2$. Suppose that there exists a point $(\bar x,\bar y)\in\mathcal O_2$ such that $\psi_1^{\tau_*}(\bar x,\bar y)>\psi_2(\bar x,\bar y)$. Since $\psi_1^{\tau_*}$ and $\psi_2$ have the same asymptotic behavior as $x\to-\infty$, and $\psi_1^{\tau_*}\leq\psi_2$ as $x\to\infty$ by the comparison principle (cf. \cite[Theorem 10.7]{GT_book} and \eqref{det_0}), one can find a domain $\mathcal O_2'\subset\mathcal O_2$ containing $(\bar x,\bar y)$ such that $\psi_1^{\tau_*}-\psi_2$ obtains its maximum in the interior of $\mathcal O_2'$. This contradicts the strong maximum principle. Hence $\psi_1^{\tau_*}\leq\psi_2$ in $\MO_2$. Noting that at the touching point one has $\psi_1^{\tau_*}(x_*,y_*)=\psi_2(x_*,y_*)$, then by the Hopf lemma it holds
$$\Ld_1=\frac1{y_*}\frac{\pt\psi_1^{\tau_*}}{\pt\nu}(x_*,y_*)>\frac1{y_*}\frac{\pt\psi_2}{\pt\nu}(x_*,y_*)=\Ld_2.$$
This leads to a contradiction. Hence one has $\Ld_1=\Ld_2$. 
This completes the proof of the lemma.
\end{proof}

Combining all the results in previous sections, Theorem \ref{result} is proved.

\bibliographystyle{plain}

\end{document}